\documentclass[a4paper,11pt]{article}

\usepackage{amsmath,amsthm,amssymb,mathrsfs,latexsym,amsfonts}
\usepackage{graphicx,psfrag,epsfig}
\usepackage{bbm}
\usepackage[english]{babel}
\usepackage[latin1]{inputenc}
\usepackage{a4wide}

\newtheorem{theorem}{Theorem}
\newtheorem{lemma}{Lemma}
\newtheorem{proposition}{Proposition}

\newtheorem{conjecture}{Conjecture\rm}

\newcounter{paraga}[section]

\begin{document}

\def\MP{\,{<\hspace{-.5em}\cdot}\,}
\def\SP{\,{>\hspace{-.3em}\cdot}\,}
\def\PM{\,{\cdot\hspace{-.3em}<}\,}
\def\PS{\,{\cdot\hspace{-.3em}>}\,}
\def\EP{\,{=\hspace{-.2em}\cdot}\,}
\def\PP{\,{+\hspace{-.1em}\cdot}\,}
\def\PE{\,{\cdot\hspace{-.2em}=}\,}
\def\N{\mathbb N}
\def\C{\mathbb C}
\def\Q{\mathbb Q}
\def\R{\mathbb R}
\def\T{\mathbb T}
\def\A{\mathbb A}
\def\Z{\mathbb Z}
\def\demi{\frac{1}{2}}

\begin{titlepage}
\author{Abed Bounemoura~\footnote{CNRS-IMPA UMI 2924  (abedbou@gmail.com)}}
\title{\LARGE{\textbf{Nekhoroshev's estimates for quasi-periodic time-dependent perturbations}}}
\end{titlepage}

\maketitle

\begin{abstract}
In this paper, we consider a Diophantine quasi-periodic time-dependent analytic perturbation of a convex integrable Hamiltonian system, and we prove a result of stability of the action variables for an exponentially long interval of time. This extends known results for periodic time-dependent perturbations, and partly solves a long standing conjecture of Chirikov and Lochak. We also obtain improved stability estimates close to resonances or far away from resonances, and a more general result without any Diophantine condition.
\end{abstract}

\section{Introduction and results}

\subsection{Introduction}

Let $n \geq 1$ be an integer, $D \subset \R^n$ an open bounded convex domain and $\T^n:=\R^n / (2\pi\Z)^n$. Consider a smooth Hamiltonian function $H$ defined on the domain $\T^n \times B$ of the form
\begin{equation}\label{aut}
H(\theta,I)=h(I)+\varepsilon f(\theta,I), \quad \varepsilon \geq 0, \quad (\theta,I)=(\theta_1,\dots,\theta_n,I_1,\dots,I_n) \in \T^n \times D, 
\end{equation} 
and its associated Hamiltonian system
\begin{equation*}
\begin{cases} 
\dot{\theta}_i(t)=\partial_{I_i} H(\theta(t),I(t))=\partial_{I_i} h(I(t))+ \varepsilon \partial_{I_i} f(\theta(t),I(t)), \\
\dot{I}_i(t)=- \partial_{\theta_i} H(\theta(t),I(t))=- \varepsilon \partial_{\theta_i} f(\theta(t),I(t)) 
\end{cases}
\quad 1 \leq i \leq n. 
\end{equation*} 
For $\varepsilon=0$, the system is stable in the sense that the action variables $I(t)$ of all solutions are constant, and all solutions are quasi-periodic. Now for $\varepsilon \neq 0$ but sufficiently small, a fundamental result of Nekhoroshev states that if the system is real-analytic and the integrable part $h$ satisfies a generic condition (called steepness), then the action variables $I(t)$ of all solutions are almost constant for an interval of time which is exponentially long with respect to the inverse of the perturbation. More precisely, the following estimates hold true along all solutions:
\begin{equation}\label{nek}
||I(t)-I(0)||:=\sqrt{\sum_{i=1}^n(I_i(t)-I_i(0))^2}\leq R_0\varepsilon^b, \quad |t|\leq T_0\exp\left(c\left(\frac{1}{\varepsilon}\right)^a\right) 
\end{equation}
for some positive constants $R_0,T_0,c,a$ and $b$. We refer to \cite{Nek77} and \cite{Nek79}. The most important constants appearing in the estimates~\eqref{nek} are undoubtedly the constants $a$ and $b$, which are called the stability exponents. The simplest class of steep integrable Hamiltonians are the convex (or quasi-convex) Hamiltonians, and these exponents depend then only on the number of degrees of freedom $n$. Nekhoroshev's original proof yielded the following dependence
\[ a \sim b \sim \frac{1}{n^2}.\]
Such values for the exponents were however much worse than the values
\[ a=b=\frac{1}{2n}\]
conjectured by Chirikov (\cite{Chi79}) on a basis of a heuristic argument and numerical simulations. This issue was later solved by Lochak (\cite{Loc92}): more precisely, Lochak-Neishtadt (\cite{LN92}) and independently Pöschel (\cite{Pos93}) proved that~\eqref{nek} holds true with the values
\[ a=b=\frac{1}{2n}. \]
More generally, it is proved in~\cite{BM11} that for any $0 \leq \delta \leq (2n(n-1))^{-1}$, one can choose
\[ a=\frac{1}{2(n-1)}-\delta, \quad b=\delta(n-1), \]
recovering the latter result by letting $\delta=(2n(n-1))^{-1}$. Examples of Arnold diffusion (\cite{Bes96}, \cite{Bes97}, \cite{Zha11}) show that in any event $a<(2(n-2))^{-1}$ so that in the convex case, the estimate~\eqref{nek} is, as far as the dependence on $a$ is concerned, quite sharp.

Now instead of an autonomous perturbation as in~\eqref{aut}, one may consider a periodic time-dependent perturbation, that is one looks at the Hamiltonian
\begin{equation}\label{nonautper1}
H(\theta,I)=h(I)+\varepsilon f(\theta,I,t), \quad \varepsilon \geq 0, \quad (\theta,I) \in \T^n \times B, \quad t \in \T. 
\end{equation}
Setting $t=\varphi \in \T$ and introducing a variable $J \in \R$ canonically conjugated to $\varphi$, it is equivalent to study the autonomous Hamiltonian
\begin{equation}\label{nonautper2}
H(\theta,I,\varphi,J)=h(I)+J+\varepsilon f(\theta,I,\varphi), \quad \varepsilon \geq 0, \quad (\theta,I) \in \T^n \times B, \quad (\varphi,J) \in \T \times \R. 
\end{equation}
Indeed, $(\theta(t),I(t),\varphi(t),J(t))$ is a solution of the system associated to~\eqref{nonautper2} if, and only if, $(\theta(t),I(t))$ is a solution of the system associated to~\eqref{nonautper1}. The Hamiltonian~\eqref{nonautper2} has now $n+1$ degrees of freedom, the integrable part $h(I)+J$ is no longer convex but it is quasi-convex: thus the results of \cite{LN92} and \cite{Pos93} apply, and the estimates~\eqref{nek} hold true for the Hamiltonian~\eqref{nonautper1} with
\[ a=b=\frac{1}{2(n+1)}. \]
Now a periodic time-dependent perturbation is nothing but a special case of a quasi-periodic time-dependent perturbation, and we may more generally consider a Hamiltonian of the form
\begin{equation}\label{nonautqper1}
H(\theta,I)=h(I)+\varepsilon f(\theta,I,t\alpha), \quad \varepsilon \geq 0, \quad (\theta,I) \in \T^n \times B, \quad t\alpha=t(\alpha_1,\dots,\alpha_m) \in \T^m 
\end{equation}   
where $\alpha \in \R^m$ is a vector which is assumed to be non-resonant, that is $k\cdot\alpha\neq 0$ for any non-zero $k \in \Z^m$ and where $\,\cdot\,$ denotes the Euclidean scalar product. We will assume actually that $\alpha$ satisfies a Diophantine condition: there exist $\gamma>0$ and $\tau \geq m-1$ such that 
\begin{equation}\label{dio}\tag{$\mathrm{Dio}_{\gamma,\tau}$}
|k\cdot\alpha|\geq \gamma|k|^{-\tau}, \quad k=(k_1,\dots,k_m) \in \Z^m\setminus\{0\}, \quad |k|:=|k_1|+\cdots |k_m|.
\end{equation} 
As before, setting $\varphi=t\alpha\in \T^m$ and introducing a vector $J=(J_1,\dots,J_m)\in \R^m$ canonically conjugated to $\varphi$, the Hamiltonian~\eqref{nonautqper1} is equivalent to
\begin{equation}\label{nonautqper2}
H(\theta,I,\varphi,J)=h(I)+\alpha\cdot J+\varepsilon f(\theta,I,\varphi), \quad \varepsilon \geq 0, \quad (\theta,I) \in \T^n \times B, \quad (\varphi,J) \in \T^m \times \R^m. 
\end{equation}   
The Hamiltonian~\eqref{nonautqper2}, in the special case $m=1$, reduces to Hamiltonian~\eqref{nonautper2}: indeed, $\alpha \in \R$ and by a scaling one may assume that $\alpha=1$, and moreover~\eqref{dio} is obviously satisfied for $\gamma=|\alpha|=1$ and $\tau=m-1=0$.

We can now state the most general form of a conjecture of Chirikov (see \cite{Chi79}, \cite{CV89} and \cite{CV96}), stated in a more precise manner by Lochak (\cite{LMS}).

\begin{conjecture}\label{conj}
The estimates~\eqref{nek} hold true for the Hamiltonian~\eqref{nonautqper2}, provided it is real-analytic and $h$ convex, with the exponents
\[ a=b=\frac{1}{2(n+1+\tau)}. \]
\end{conjecture}

Note that this conjecture is made plausible by the fact that in the periodic case, that is $m=1$ and $\tau=0$, it is a theorem. Yet for $m \geq 2$ and hence $\tau \geq 1$, it is an open question whether estimates~\eqref{nek} hold true for the Hamiltonian~\eqref{nonautqper2}, for some values of $a$ and $b$ whatsoever. 

It is the purpose of this article to solve this problem: we will prove that the estimates~\eqref{nek} hold true for the Hamiltonian~\eqref{nonautqper2}, provided it is real-analytic and $h$ convex, with the exponents
\begin{equation}\label{exposant}
a=\frac{1}{2(n+1)(1+\tau)}, \quad b=\frac{(n+1)\tau+1}{2(n+1)(1+\tau)}=a+\frac{\tau}{2(1+\tau)}.
\end{equation}
We refer to Theorem~\ref{thm1} below for a more precise statement. Concerning the values of the exponents we obtain, let us just make two comments (a more detailed discussion is contained in Section~\S\ref{s8}). First, for the periodic case ($m=1$, $\tau=0$), we also recover the known values of the exponents, so our result can also be considered as a ``correct" generalization of the periodic case. Then, in the non-periodic case ($m \geq 2$, $\tau \geq 1$), our exponent $a$ is substantially worse than the one of Conjecture~\ref{conj}, but at the same time our exponent $b$ is always much better: $b$ is close to $1/2$ (it is always strictly bigger than $1/4$), and when $n$ is fixed and $m$ (or $\tau$) becomes arbitrary large, it becomes arbitrarily close to $1/2$. Moreover, in the non-periodic case, $b$ is essentially independent of $n$ which is an interesting feature. In any case, it is therefore still an open question whether the values of the exponents of Conjecture~\ref{conj} can be reached or not.

In fact, we do obtain more general results. First, the perturbation will be allowed to depend also one the $J$ variables, that is we can replace $f(\theta,I,\varphi)$ in~\eqref{nonautqper2} by $f(\theta,I,\varphi,J)$, provided $f$ is bounded and real-analytic in $J$, when $J$ varies in $\R^m$. This will be the precise content of Theorem~\ref{thm1}. To understand the interest of this seemingly mild extension, one can compare such a result with the preservation of invariant tori (that is, the KAM theory) for Hamiltonians as in~\eqref{aut},~\eqref{nonautper2} or~\eqref{nonautqper2}. In the autonomous case of~\eqref{aut}, if the integrable Hamiltonian is convex then it is in particular Kolmogorov non-degenerate and the classical KAM theory applies (see \cite{Pos01} for a survey). In the periodic case~\eqref{nonautper2}, the integrable Hamiltonian is no longer Kolmogorov non-degenerate but as it is quasi-convex, it is Arnold (or iso-energetically) non-degenerate: it follows that tori are preserved at a fixed energy for~\eqref{nonautper2} yielding invariant tori for~\eqref{nonautper1}. In this case the perturbation may also depend on $J\in \R$ without affecting the result. Now in the more general quasi-periodic case~\eqref{nonautqper2}, the integrable part is both Kolmogorov and Arnold degenerate. However, using the non-degeneracy with respect to the $I$ variables and the fact that the perturbation is independent of $J$, it is not hard to prove, using classical KAM techniques, that many tori with prescribed Diophantine frequencies of the form $(\omega,\alpha)\in \R^{n+m}$ are preserved (up to our knowledge, this was first observed by Galavotti in~\cite{Gal94} in a restricted situation and later by Lochak in \cite{Loc98}, \cite{LMS} in a general situation). It is crucial here to have a perturbation which is independent of $J$: if not, the method simply breaks down and it is rather easy to construct counter-examples (as in \cite{Sev03}) to the preservation of (full dimensional) invariant tori. Therefore unlike the situation in KAM theory, our result is exactly the same when $f$ is allowed to depend on the $J$ variables.

Then, as in the autonomous or periodic case, we do obtain enhanced stability close to resonances. In our situation, the multiplicity of any resonance is at most $d$, where $0 \leq d \leq n$ (with the convention that any frequency is resonant of multiplicity at least $0$), and solutions who start sufficiently close to such a resonance satisfy~\eqref{nek} with the exponents
\[ a(d)=\frac{1}{2((n+1)\tau+n+1-d)} , \quad b(d)=\frac{(n+1)\tau+1}{2((n+1)\tau+n+1-d)}.\]
In the case of a resonance of maximal multiplicity $d=n$ these exponents read
\[ a(n)=\frac{1}{2((n+1)\tau+1)} , \quad b(n)=\frac{1}{2}.\]
This will be proved in Theorem~\ref{thm2}. One recovers the improved stability exponents of the periodic case by setting $\tau=0$, and our main result by setting $d=0$. 
 
The complement of the neighborhoods of all resonances is the non-resonant domain (this domain contains, in particular, invariant tori, if any). Solutions starting in the non-resonant domain satisfy~\eqref{nek} with the exponents
\[ a=\frac{1}{2(n+1)(\tau+1)} , \quad b'=\frac{1}{2}.\] 
Moreover, the complement of this non-resonant domain is actually very small: its measure is of order 
\[ \varepsilon^{b}, \quad b=\frac{(n+1)\tau+1}{2(n+1)(1+\tau)}\] 
and therefore goes to zero with $\varepsilon$. This will be stated as Theorem~\ref{thm3}. This actually improves on the corresponding statement for $\tau=0$, where the measure estimate of the complement is just of order one. This non-resonant domain comes from the proof of our main theorem, and ends up quite large as we need to exclude very small neighborhoods of resonances. A more natural definition of a non-resonant domain yield the following result: given any $0 < \gamma' \leq \gamma$ and any $\tau'$ such that $\tau'>n+m-1$ and $\tau'\geq \tau$, where $\gamma$ and $\tau$ are the constants appearing in~\eqref{dio}, there is a set whose complement has a measure of order $\gamma'$, such that on this set, the estimates~\eqref{nek} hold true with the exponents
\[ a'=\frac{1}{2(\tau'+1)} , \quad b'=\frac{1}{2}.\] 
This will be the content of Theorem~\ref{thm33}. Hence on a smaller non-resonant subset, but which is still relatively large, we have a stronger stability result, with $a'$ arbitrarily close to the value conjectured and $b'$ much better. As a matter of fact, when $\tau>m-1$ (for $m \geq 2$, the set of vectors $\alpha$ for which $\tau=m-1$ has zero measure), one can choose $\tau'=n+\tau$ and then $a'$ coincides with the value conjectured.
 
Finally, as usual with results in Hamiltonian perturbation theory concerning long but finite time scale, the Diophantine condition~\eqref{dio} on the vector $\alpha \in \R^m$ turns out to be unnecessary. For an arbitrary vector $\alpha \in \R^m$ which is assume to be non-resonant, that is
\[ k\cdot\alpha\neq 0, \quad k\neq 0 \in \Z^m, \]  
we will obtain in Theorem~\ref{thm4} a more general stability result which reduces to the main result in the case where $\alpha$ is Diophantine. As a matter of fact, even the assumption that $\alpha$ is non-resonant is unnecessary: indeed, $\alpha$ being fixed, if it is resonant we can write, without loss of generality, $\alpha=(\bar{\alpha},0)\in \R^k \times \R^{m-k}=\R^m$ for some $1 \leq k \leq m-1$, with $\bar{\alpha} \in \R^k$ non-resonant and the above applies immediately with $\alpha \in \R^m$ replaced by $\bar{\alpha} \in \R^k$.

\subsection{Main result}

Let us now state more precisely our main result. We consider a Hamiltonian of the form
\begin{equation}\label{Ham}\tag{H}
H(\theta,\varphi,I,J)=h(I)+\alpha\cdot J+f(\theta,\varphi,I,J), \quad (\theta,I) \in \T^n \times D, \quad (\varphi,J) \in \T^m \times \R^m 
\end{equation}  
where $h$ is the integrable part and $f$ the perturbation. The Hamiltonian $\bar{h}$, defined on $\bar{D}:=D \times \R^m$ by
\[ \bar{h}(I,J):=h(I)+\alpha\cdot J, \quad (I,J)\in \bar{D} \]
will be called the extended integrable part. The functions $h$ and $f$ are assumed to be real-analytic as follows. Given two parameters $r_0>0$ and $s_0>0$, we define the complex domains
\[ V_{r_0}D:=\{I \in \C^n \; | \; ||I-D||<r_0 \}, \quad V_{r_0}\bar{D}:=\{(I,J) \in \C^{n+m} \; | \; ||(I,J)-\bar{D}||<r_0 \} \]
where
\[ ||I-D||:=\inf_{I' \in D}||I-I'||, \quad ||(I,J)-\bar{D}||:=\inf_{(I',J') \in \bar{D}}||(I,J)-(I',J')||,  \]
and
\[ V_{s_0}\T^{n+m}:=\{(\theta,\varphi) \in \C^{n+m}/(2\pi\Z)^{n+m} \; | \; \max_{1 \leq i \leq n}|\mathrm{Im}(\theta_i)|<s_0, \max_{1 \leq i \leq m}|\mathrm{Im}(\varphi_i)|<s_0 \}. \] 
Let us also define the associated real domains
\[ U_{r_0}D:=V_{r_0}D \cap \R^n, \quad U_{r_0}\bar{D}:=V_{r_0}\bar{D} \cap \R^{n+m}=U_{r_0}D \times \R^m. \]
The function $h$ is assumed to be real-analytic on $V_{r_0}D$ so that $\bar{h}$ is real-analytic on $V_{r_0}\bar{D}$, and its Hessian $\nabla^2h$ is assumed to be uniformly bounded on the complex domain $V_{r_0}D$, namely there exists $M >0$ such that
\begin{equation}\label{bound}\tag{$M$}
\sup_{I \in V_{r_0}D}||\nabla^2h(I)||=\sup_{(I,J) \in V_{r_0}\bar{D}}||\nabla^2\bar{h}(I,J)||\leq M
\end{equation}
where the matrix norm is the one induced by the Euclidean norm. The gradient of $h$ is also assumed to be uniformly bounded on the real domain $U_{r_0}D$, that is there exists $\Omega >0$ such that
\begin{equation}\label{bound2}\tag{$\Omega$}
\sup_{I \in U_{r_0}D}||(\nabla h (I),\alpha)||=\sup_{I \in U_{r_0}\bar{D}}||\nabla \bar{h} (I)|| \leq \Omega. 
\end{equation}
Moreover, the integrable Hamiltonian is assumed to be (strictly, uniformly) convex: there exists $m>0$ such that for any $v \in \R^n$, 
\begin{equation}\label{convex}\tag{$m$}
\nabla^2h(I)v\cdot v \geq m ||v||^2.
\end{equation}
Observe that $m \leq M$, and, without loss of generality, we may assume that $m \leq 1$.

Finally, the function $f$ is real-analytic on $V_{r_0}\bar{D} \times V_{s_0}\T^{n+m}$, and moreover, given a small parameter $\varepsilon \geq 0$, it is assumed that
\begin{equation}\label{pert}\tag{$\varepsilon$}
|f|_{r_0,s_0} \leq \varepsilon 
\end{equation}
where the Fourier norm $|f|_{r_0,s_0}$ of $f$ is defined as follows: letting 
\[ f(\theta,\varphi,I,J)=\sum_{(k,l) \in \Z^{n+m}}f_{k,l}(I,J)e^{i (k,l)\cdot(\theta,\varphi)}\] 
be the Fourier expansion of $f$ with respect to $(\theta,\varphi)$, we define
\[ |f|_{r_0,s_0}:=\sup_{(I,J) \in V_{r_0}\bar{D}}\sum_{(k,l) \in \Z^{n+m}}|f_{k,l}(I,J)|e^{|(k,l)|s_0}, \quad |(k,l)|=|k|+|l|. \]
We can now state our main theorem.

\begin{theorem}\label{thm1}
Let $H$ be as in~\eqref{Ham}, with $h$ satisfying~\eqref{bound},~\eqref{bound2} and~\eqref{convex}, and $f$ satisfying~\eqref{pert}. Assume also that $\alpha$ satisfies~\eqref{dio}, and let us define
\[ a=\frac{1}{2(n+1)(\tau+1)}, \quad b=\frac{(n+1)\tau+1}{2(n+1)(\tau+1)} \]
and
\[ R_*=\frac{m\gamma}{10M(n+1)^\tau}, \quad T_*=\frac{3s_0}{\Omega}, \quad \varepsilon_0=\frac{m\gamma^2}{2^{10}(n+1)^{2\tau}}\left(\frac{m}{10M}\right)^{2(n+1)}, \quad \varepsilon_*=\varepsilon_0\left(\frac{r_0}{R_*}\right)^{\frac{1}{b}}. \] 
If $\varepsilon \leq \min\{\varepsilon_0,\varepsilon_*\}
$, for any solution $(I(t),J(t),\theta(t),\varphi(t))$ of the system associated to $H$ with initial condition $(I_0,J_0,\theta_0,\varphi_0) \in \bar{D} \times \T^{n+m}$, we have
\[ ||I(t)-I_0||\leq R(\varepsilon):=R_*\left(\frac{\varepsilon}{\varepsilon_0}\right)^b, \quad |t|\leq T(\varepsilon):=T_*\exp\left(\frac{s_0}{6}\left(\frac{\varepsilon_0}{\varepsilon}\right)^a\right). \]
\end{theorem}

Let us just make one simple comment concerning the assumption~\eqref{pert}, which requires the perturbation $f(\theta,\varphi,I,J)$ to be uniformly bounded in $J$, with $J$ belonging to the $r_0$-neighborhood of $\R^m$ in $\C^m$. When $H$ comes from a time-dependent quasi-periodic perturbation of a convex integrable Hamiltonian, the perturbation is in fact independent of $J$ so the above requirement is void. Then, if~\eqref{pert} is satisfied but only for a certain domain $B$ properly contained in $\R^m$, the statement remains true up to the following modification: the stability time $T(\varepsilon)$ should be replaced by $\min\{T_0,T(\varepsilon)\}$ where $T_0$ is the first (possibly infinite) exit time of $J(t)$ from the domain $U_{r_0}B$. This will follow easily from the proof of Theorem~\ref{thm1}, and Theorem~\ref{thm1} is actually a special case as $T_0=\infty$ when $B=\R^m$.

\subsection{Strategy of the proof and plan of the paper}

There are two known methods to prove Nekhoroshev type estimates for small perturbations of integrable Hamiltonian systems. The first one is the Nekhoroshev-Pöschel's method, introduced in the seminal work of Nekhoroshev (\cite{Nek77},\cite{Nek79}) and later improved by Pöschel  (\cite{Pos93}) in the convex case (see also \cite{GCB15} for a further extension of the work of Nekhoroshev and Pöschel leading to an improved value of the stability exponents in the steep case, generalizing the known values of the convex case). The second method is the Lochak method, introduced by Lochak (\cite{Loc92},\cite{LN92}) in the convex case (see also \cite{BN12} for an extension to the steep case, though with worse values for the stability exponents). In the convex case, the latter method is undoubtedly the simplest and most elegant way to prove stability estimates.

The Lochak method crucially relies on the existence of periodic orbits for the integrable system, that is on the existence of periodic frequencies. Now in the case of a quasi-periodic perturbation, the space of frequencies is of the form $(\omega,\alpha)\in \R^{n+m}$, where $\omega \in \R^n$ is free but $\alpha \in \R^{m}$ fixed and non-resonant (in particular, $\alpha$ is not periodic; if it were, one would be in fact looking at a time-dependent periodic perturbation). The issue is that this space does not contain periodic frequencies. As a matter of fact, it is not really necessary to have exact periodic frequencies, but only frequencies which can be approximated by periodic ones. Now any frequency of the form $(\omega,\alpha)$ can be approximated by a periodic frequency, say $(\omega',\alpha')$. But then necessarily $\alpha' \neq \alpha$, and since our Hamiltonian is not convex in the $J$ variables we were not able to prove stability close to such periodic frequencies. Therefore the method of Lochak does not seem to extend in a easy way to the case of a quasi-periodic perturbation. 

Our strategy is therefore to try to extend the Nekhoroshev-Pöschel method, and we will succeed in doing so. This method consists of covering the space of frequencies by resonant blocks, which are neighborhood of resonances (defined by a certain lattice of integer vectors) that are otherwise non-resonant (for integer vectors not in the lattice). Using convexity, one can then show that the solutions stay in the same resonant block for a long time (in the general steep case, they may leave their resonant block and the proof of the stability gets much harder). In the autonomous case where the frequency space is just $\omega \in \R^n$, resonances define linear subspaces which are orthogonal to arbitrary submodules of $\Z^n$. In the quasi-periodic case, our frequency space is still $n$-dimensional but resonances are associated to submodules of $\Z^{n+m}$. But not all submodules of $\Z^{n+m}$ are associated to resonances: those that are not will be called non admissible. Now resonances associated to admissible submodules do not necessarily define linear subspaces but rather affine subspaces in the space of $\omega \in \R^n$. Therefore we are facing much more resonances than in the autonomous case, and the geometry of these resonances gets more involved. In particular, different admissible submodules might lead to different but parallel affine subspaces, and it is at this point that the assumption that $\alpha$ is Diophantine (or in fact simply non-resonant) comes into play: it ensures that we can control the distance between these parallel affine subspaces. In particular, in the extreme case where these parallel affine subspaces are just points (that is, their associated vector space is trivial), they can get very close to each other, and this is precisely from this phenomenon that our values of the stability exponents come from. This improved geometry of resonances will lead to the fact that any frequency $(\omega,\alpha)$ is close to some resonant frequency of the form $(\omega',\alpha)$. Using this, and the fact that our integrable Hamiltonian is convex in the $I$ variables while linear in the $J$ variables, we will be able to prove stability for the $I$ variables (but we will loose control on the evolution of the $J$ variables). 

Now the plan of the paper is as follows. Section~\S\ref{s2} deals in details with the geometry of resonances that was alluded above. This section contains the main technical part of the work. Section~\S\ref{s3} deals with the analysis (construction of a normal form) and the local stability results. The analysis, and therefore the stability in the non-resonant case, is completely standard and we can simply refer to \cite{Pos93}. The stability in the resonant case uses convexity: our integrable Hamiltonian is just ``partially" convex so we need to justify that the arguments go through, at the expense of loosing control on the $J$ variables. The proof of our main result Theorem~\ref{thm1} will be given in Section~\S\ref{s4}, using the results of Section~\S\ref{s2} and Section~\S\ref{s3}. The next sections contain further results that were mentioned in the Introduction:  namely, we prove better stability results for solutions close to resonances in~\S\ref{s5} or far way from resonances in~\S\ref{s6}, while in~\S\ref{s7} we give a more general result assuming $\alpha$ to be only non-resonant. The last Section~\S\ref{s8} consists of concluding remarks.

\section{Geometry of resonances}\label{s2}

The purpose of this section is to study the resonant and non-resonant properties of the frequency space
\[ \{(\omega,\alpha) \in \R^{n+m}\} \simeq \{\omega \in \R^n\} \]
where $\alpha \in \R^{m}$ is a fixed vector, which will be assumed to be Diophantine, and $\omega$ is a vector varying freely in $\R^{n}$. More precisely, our aim is to cover this space by neighborhoods of resonances (associated to certain submodules $\Lambda$ of $\Z^{n+m}$) on which non-resonant estimates can be established (for integer vectors $k \notin \Lambda$).

\subsection{Admissible resonant zones and resonant blocks}\label{s21}

We fix a real parameter $K \geq 1$, and in this section, $\alpha \in \R^m$ will be assumed to be simply non-resonant. 

A submodule $\Lambda$ of $\Z^{n+m}$ is said to be a $K$-submodule if it is generated by elements $(k,l) \in \Z^n \times \Z^m=\Z^{n+m}$ such that $|(k,l)|\leq K$, and it is said to be maximal if it is not properly contained in any other submodule of the same dimension. Given an integer $1 \leq d \leq n+m$, the set of all maximal K-submodules $\Lambda$ of $\Z^{n+m}$ of rank $d$ will be denoted by $M_{K,d}$. For $\Lambda \in M_{K,d}$, we define the space of $\Lambda$-resonances by
\begin{equation}\label{res}
R_\Lambda=\{\omega \in \R^n \; | \; (k,l)\cdot (\omega,\alpha)=0, \; \forall (k,l) \in \Lambda\}.
\end{equation} 
Quite obviously, since $\alpha \in \R^m$ is non-resonant, $R_{\Lambda}$ will be non-empty only for certain maximal $K$-submodules $\Lambda$. Let us consider the subset $M_{K,d}^a$ of $M_{K,d}$ consisting of admissible submodules: they are submodules whose intersection with $\{0\}\times \Z^m \subset \Z^n \times \Z^m=\Z^{n+m}$ is trivial. Equivalently, given any basis $(k^1,l^1),\dots,(k^d,l^d)$ for $\Lambda$, the vectors $k^1,\dots,k^d$ in $\Z^n$ are linearly independent. It is plain to check that if $\Lambda$ is not admissible, then $R_\Lambda$ is just the empty set. Note also that if $\Lambda$ is admissible, its rank is at most $n$.

Now consider $\Lambda \in M_{K,d}^a$ where $1 \leq d \leq n$. If $\Pi: \R^{n+m} \mapsto \R^n$ is the canonical projection, $\tilde{\Lambda}:=\Pi(\Lambda)$ is a submodule of $\Z^n$, of rank $d$, which generates a real subspace $\langle \tilde{\Lambda} \rangle$ of $\R^n$ of dimension $d$. It is clear that $\tilde{\Lambda}$ is a $K$-submodule, but it is not necessarily maximal. The space of $\Lambda$-resonances defined in~\eqref{res} is non-empty, it is an affine subspace of $\R^n$ whose associated vector subspace is the vector subspace $\langle \tilde{\Lambda} \rangle^\perp$ orthogonal to $\langle \tilde{\Lambda} \rangle$.

It is not the space of resonances but rather their neighborhoods that will play a role in the construction below. To define them, given $\Lambda \in M_{K,d}^a$ and its associated submodule $\tilde{\Lambda}$, we first define $|\tilde{\Lambda}|$ as the co-volume of $\tilde{\Lambda}$ viewed as a lattice in $\langle \tilde{\Lambda} \rangle$. It is the volume of the fundamental domain spanned by the vectors of any choice of basis for $\tilde{\Lambda}$: letting $A$ be an $n \times d$ matrix whose columns form a basis for $\tilde{\Lambda}$, then $|\tilde{\Lambda}|=\sqrt{\mathrm{det}A^{t}A}$, and this latter quantity is easily seen to be independent of the choice of a basis. It is worth recalling, as it will be used, that if $S_d(A)$ denotes all square matrices of size $d$ that can be extracted from $A$, then we have the equality (Cauchy-Binet formula)
\[ |\tilde{\Lambda}|=\sqrt{\mathrm{det}A^{t}A}=\sqrt{\sum_{B \in S_d(A)}(\mathrm{det}B)^2}. \]
Then, we introduce $n$ positive real parameters $\lambda_1,\lambda_2,\dots,\lambda_{n}$ and, for $1 \leq d \leq n$, we define the associated resonant zone
\begin{equation}\label{reszone}
Z_\Lambda:=\{\omega \in \R^n \; | \; ||\omega-R_{\Lambda}||<\delta_\Lambda\}, \quad \delta_\Lambda:=\frac{\lambda_d}{|\tilde{\Lambda}|}
\end{equation}
where 
\[ ||\omega-R_{\Lambda}||:=\inf_{\omega' \in R_{\Lambda}}||\omega-\omega'||. \]
We then define the resonant zone of multiplicity $d$, for $1 \leq d \leq n+1$, by
\[ Z_d:=\bigcup_{\Lambda \in M_{K,d}^a}Z_\Lambda, \quad 1 \leq d \leq n, \quad Z_{n+1}=\emptyset. \]
The resonant block associated to $\Lambda \in M_{K,d}^a$, $1 \leq d \leq n$, are defined by
\[ B_\Lambda:=Z_\Lambda \setminus Z_{d+1} \]
and eventually the resonant block of multiplicity $d$, for $1 \leq d \leq n$, is
\[ B_d:=\bigcup_{\Lambda \in M_{K,d}^a}B_\Lambda. \]
Setting $B_{\{0\}}=\R^n\setminus Z_{1}$, we arrive at the following decomposition
\begin{equation}\label{decomp1}
\R^n=B_{\{0\}} \cup Z_1=B_{\{0\}} \cup B_1 \cup Z_2=\cdots=B_{\{0\}} \cup B_1 \cup \dots B_{n-1} \cup B_{n}
\end{equation}
since $B_n=Z_n$ as $Z_{n+1}=\emptyset$.

\subsection{Non-resonant domains in frequency space}\label{s22}

Consider a domain $B \subset \R^n$, a submodule 
\[ \Lambda \in M_{K}^a:=\bigcup_{1 \leq d \leq n}M_{K,d}^a \cup \{0\}\] 
and a real parameter $\beta>0$. Then the domain $B$ is said to be $(\beta,K)$-non resonant modulo $\Lambda$ if for any $(k,l) \in \Z^{n+m}$ such that $|(k,l)|\leq K$ and $(k,l) \notin \Lambda$ and any $\omega \in B$, we have
\[ |(k,l)\cdot(\omega,\alpha)|\geq \beta. \]
Our purpose here is to show that the resonant blocks $B_\Lambda$, for $\Lambda \in M_{K}^a$, are $(\beta_\Lambda,K)$-non resonant modulo $\Lambda$, for a suitable $\beta_\Lambda$ provided that $\alpha$ is Diophantine and the parameters $\lambda_1,\lambda_2,\dots,\lambda_n$ satisfies certain compatibility conditions. This is the content of the lemma below. 

\begin{lemma}\label{nonresonant}
Let $E>0$ and $F \geq E+1$. Assume that $\alpha \in \R^m$ satisfies~\eqref{dio} and 
\begin{equation}\label{lambda}
\begin{cases}
FK\lambda_d \leq \lambda_{d+1} \leq \gamma (d+1)^{-\tau} K^{-(d+1)\tau}, \quad 1 \leq d \leq n-1, \\ 
\lambda_n \leq F^{-1}\gamma (n+1)^{-\tau} K^{-(n+1)\tau-1}.
\end{cases}
\end{equation}
Then for any $\Lambda \in M_{K}^a$, the block $B_\Lambda$ is $(\beta_\Lambda,K)$-non resonant modulo $\Lambda$ with
\begin{equation}\label{beta}
\beta_\Lambda= EK\delta_{\Lambda}, \quad \Lambda\neq\{0\}, \quad \beta_{\{0\}}=\lambda_1.
\end{equation}
\end{lemma}

\begin{proof}
Let $\Lambda \in M_{K}^a$ of rank $d$, with $0 \leq d \leq n$, and $(k,l)\notin \Lambda$ such that $|(k,l)|\leq K$. Let $\Lambda_+$ be the submodule of $\Z^{n+m}$ generated by $\Lambda$ and $(k,l)$. Since $\Lambda$ is maximal and does not contain $(k,l)$, the rank of $\Lambda_+$ is equal to $d+1$. 

Let us start with the special case $d=0$, that is $\Lambda=\{0\}$, and fix $\omega \in B_{\{0\}}$. Either $\Lambda_+$ is admissible, or not. In the second situation, $k=0\in \Z^n$ while $l\neq 0 \in \Z^m$ as $(k,l)\notin \Lambda=\{0\}$, and we have, using the fact that $\alpha$ satisfies~\eqref{dio} and~\eqref{lambda}, 
\begin{equation}\label{bout1}
|(k,l)\cdot(\omega,\alpha)|=|l\cdot \alpha|\geq \gamma K^{-\tau} \geq \lambda_1.
\end{equation} 
In the first situation where $k\neq 0 \in \Z^{n}$, let $\omega_+ \in R_{\Lambda_+}$ such that
\[ ||\omega-\omega_+||=||\omega-R_{\Lambda_+}||>\delta_{\Lambda_+}. \]
Now as $(k,l)\cdot(\omega_+,\alpha)=0$ we have
\[ (k,l)\cdot(\omega,\alpha)=(k,l)\cdot(\omega,\alpha)- (k,l)\cdot(\omega_+,\alpha)=k\cdot(\omega-\omega_+).\]
But the vector $\omega-\omega_+$ belongs to the line orthogonal to $\langle \tilde{\Lambda}_+ \rangle^\perp=\langle k \rangle^\perp$ in $\langle \tilde{\Lambda} \rangle^\perp=\R^n$ which is nothing but the line generated by $k$, so we obtain
\begin{equation}\label{bout2}
|(k,l)\cdot(\omega,\alpha)|=|k\cdot(\omega-\omega_+)|=||k||||\omega-\omega_+||>||k||\delta_{{\Lambda}_+}=||k|||\tilde{\Lambda}_+|^{-1}\lambda_1\geq \lambda_1 
\end{equation}
where we used the fact $|\tilde{\Lambda}_+|\leq ||k||$. From~\eqref{bout1} and~\eqref{bout2} the statement in the case $d=0$ follows. 

Now assume $1 \leq d \leq n$. It is enough to prove that given any $\omega \in R_{\Lambda} \setminus Z_{d+1}$ (where we recall that $Z_{n+1}=\emptyset$), we have
\begin{equation}\label{divisors}
|(k,l)\cdot (\omega,\alpha)| \geq FK\delta_{\Lambda}.
\end{equation}
Indeed, for any $\bar{\omega}\in B_\Lambda=Z_\Lambda \setminus Z_{d+1}$, by definition there exists $\omega \in R_{\Lambda} \setminus Z_{d+1}$ such that $||\omega-\bar{\omega}||<\delta_\Lambda$ and thus
\begin{eqnarray*}
|(k,l)\cdot (\bar{\omega},\alpha)| & \geq & |(k,l)\cdot (\omega,\alpha)|-|k\cdot(\bar{\omega}-\omega)| \\
& \geq & |(k,l)\cdot (\omega,\alpha)|-||k||||(\bar{\omega}-\omega)|| \\
& \geq & |(k,l)\cdot (\omega,\alpha)|-K||(\bar{\omega}-\omega)|| \\
& \geq & FK\delta_{\Lambda}-K\delta_{\Lambda}=(F-1)K\delta_{\Lambda} \geq EK\delta_{\Lambda}.
\end{eqnarray*}
So it suffices to prove~\eqref{divisors}. As before, there are two possibilities for $\Lambda_+$: either it is admissible, or not. In the second situation (which is obviously the only possibility for $d=n$), let us choose a basis $(k^1,l^1), \dots, (k^d,l^d)$ for $\Lambda$ such that $|(k^j,l^j)|\leq K$ for $1 \leq j \leq d$. The assumption that $\Lambda_+$ is not admissible means that $k$ is a linear combination of $k^1, \dots, k^d$. For $1 \leq j \leq d$, let us define $\bar{k}^j \in \Z^d$ by selecting the first $d$ components of $k^j$, and we define $\bar{k}$ the same way. As $k^1, \dots, k^d$ are linearly independent (because $\Lambda$ is admissible), the determinant $\Delta(\bar{k}^1,\dots,\bar{k}^d)$ of the square matrix of size $d$ whose columns are given by $\bar{k}^1,\dots,\bar{k}^d$ can be assumed to be non-zero without loss of generality. By Cramer's rule, it then follows that $k$ can be written as
\[ k=\nu_1k^1+\nu_2k^2+\cdots+\nu_dk^d=\frac{\Delta_1}{\Delta}k^1+\frac{\Delta_2}{\Delta}k^2+\cdots+\frac{\Delta_d}{\Delta}k^d \]
where $\Delta=\Delta(\bar{k}^1,\dots,\bar{k}^d)$, $\Delta_1=\Delta(\bar{k},\bar{k}^2\dots,\bar{k}^d)$, $\Delta_d=\Delta(\bar{k}^1,\bar{k}^2\dots,\bar{k})$ and for $2 \leq j \leq d-1$, $\Delta_j=\Delta(\bar{k}^1,\dots,\bar{k}^{j-1},\bar{k},\bar{k}^{j+1}, \dots,\bar{k}^d)$. Let us denote
\[ l_*=l-\nu_1l^1-\cdots-\nu_dl^d \in \Z^m. \]
The vector $l_*$ is non-zero: if it were, $(k,l)$ would be a linear combination of $(k^1,l^1), \dots, (k^d,l^d)$, that is $(k,l) \in \langle \Lambda \rangle \cap \Z^{n+m}$, but since $\Lambda$ is maximal it is equal to $\langle \Lambda \rangle \cap \Z^{n+m}$ and so this would contradict $(k,l) \notin \Lambda$. Now using the fact that $(k^j,l^j)\in \Lambda$ for $1 \leq j \leq d$ and $\omega \in R_\Lambda$ we can write
\[ (k,l)\cdot(\omega,\alpha)=(k,l)\cdot(\omega,\alpha)-\sum_{j=1}^d\nu_j(k^j,l^j)\cdot(\omega,\alpha)=(l_*,\alpha). \]
The vector $\Delta l_*=\Delta l-\Delta_1l^1-\cdots -\Delta_dl^d \in \Z^m$, and by Hadamard's formula, we have \[|\Delta| \leq ||\bar{k}^1||\cdots ||\bar{k}^d|| \leq K^d \] 
and similarly $|\Delta_j| \leq K^d$ for $1 \leq j \leq d$, so as a consequence
\[ |\Delta l_*|\leq |\Delta||l|+|\Delta_1||l^1|+\cdots+|\Delta_d||l^d|\leq (d+1)K^{d+1}. \]
Since $\alpha$ satisfies~\eqref{dio}, it follows that
\[ |\Delta||(k,l)\cdot(\omega,\alpha)|=|\Delta||(l_*,\alpha)|=|(\Delta l_*,\alpha)| \geq \gamma (d+1)^{-\tau}K^{-(d+1)\tau} \]
and as a consequence of Cauchy-Binet formula, $|\Delta|\leq |\tilde{\Lambda}|$ and hence
\[ |(k,l)\cdot(\omega,\alpha)|\geq |\tilde{\Lambda}|^{-1}\gamma (d+1)^{-\tau}K^{-(d+1)\tau}. \]
Using this last inequality together with~\eqref{lambda} one arrives at 
\begin{equation}\label{bout3}
|(k,l)\cdot(\omega,\alpha)|\geq |\tilde{\Lambda}|^{-1}\lambda_{d+1}\geq |\tilde{\Lambda}|^{-1} FK\lambda_d=FK\delta_\Lambda.
\end{equation}
It remains to treat the case where $\Lambda_+$ is admissible. Let $\omega_+ \in R_{\Lambda_+}$ such that
\[ ||\omega-\omega_+||=||\omega-R_{\Lambda_+}||>\delta_{\Lambda_+}. \]
Then since $(k,l)\cdot (\omega_+,\alpha)=0$ we have
\[ (k,l)\cdot(\omega,\alpha)=k\cdot(\omega-\omega_+). \]
The vector $\omega-\omega_+$ belongs to the line orthogonal to $\langle \tilde{\Lambda}_+ \rangle^\perp$ within $\langle \tilde{\Lambda} \rangle^\perp$, that is it belongs to $L:=\langle \tilde{\Lambda}_+ \rangle^{\perp\perp} \cap \langle \tilde{\Lambda} \rangle^\perp=\langle \tilde{\Lambda}_+ \rangle \cap \langle \tilde{\Lambda} \rangle^\perp$. Let us decompose $k=Pk+(\mathrm{Id}-P)k$ where $Pk \in L$ and $(\mathrm{Id}-P)k \in L^\perp$. If follows that
\[ |(k,l)\cdot(\omega,\alpha)|=|k\cdot(\omega-\omega^+)|=|Pk\cdot(\omega-\omega_+)|=||Pk||||\omega-\omega_+||>||Pk||\delta_{\Lambda_+}. \]  
But $L$ is also the line orthogonal to $\langle \tilde{\Lambda} \rangle$ within $\langle \tilde{\Lambda}_+ \rangle$, and therefore $|\tilde{\Lambda}_+|\leq ||Pk|||\tilde{\Lambda}|$ so, using~\eqref{lambda} again,
\begin{equation}\label{bout4}
|(k,l)\cdot(\omega,\alpha)| > |\tilde{\Lambda}_+||\tilde{\Lambda}|^{-1}\delta_{\Lambda_+}=|\tilde{\Lambda}|^{-1}\lambda_{d+1} \geq |\tilde{\Lambda}|^{-1}FK\lambda_{d}=FK\delta_\Lambda. 
\end{equation}
From~\eqref{bout3} and~\eqref{bout4} it follows that~\eqref{divisors} holds true for $1 \leq d \leq n$, and this concludes the proof.

\end{proof}

\subsection{Covering by non-resonant domains in action space}\label{s23}

In section \S\ref{s21} we obtained a covering of the frequency space $\{(\omega,\alpha) \in \R^{n+m}\} \simeq \{ \omega \in \R^n\}$ by resonant blocks associated to admissible submodules, and in Lemma~\ref{nonresonant} we proved that these resonant blocks satisfy some non-resonant properties.

Recall that we are given an integrable Hamiltonian $h$ which is real-analytic on the domain $V_{r_0}D$, and satisfies~\eqref{bound} and~\eqref{convex}. A subset of $D$ is said to be $(\beta,K)$-non resonant modulo $\Lambda$ for $h$ if its image by the frequency map $\nabla h$ is $(\beta,K)$-non resonant modulo $\Lambda$. It is said to be $\delta$-close to $\Lambda$-resonances for $h$ if the Euclidean distance between its image by $\nabla h$ and the space $R_\Lambda$ is smaller than $\delta$.

Pulling back the covering~\eqref{decomp1} back to action space using the gradient map $\nabla h$ and using Lemma~\ref{nonresonant}, the following proposition will be easily obtained by carefully choosing the parameters $\lambda_d$, for $1 \leq d \leq n$.

\begin{proposition}\label{covering}
For $\Lambda \in M_K^a$ of rank $d$, with $0 \leq d \leq n$, let us define
\begin{equation*}
r_\Lambda:=\frac{\gamma}{|\tilde{\Lambda}|F^{n-d+1}(n+1)^\tau K^{(n+1)\tau+n-d+1}},
\end{equation*}
where we set, by convention, $|\tilde{\Lambda}|=1$ if $\Lambda=\{0\}$. Then there exists a covering of $D$ by subsets $D_{\Lambda}$, where $\Lambda \in M_K^a$, such that each $D_\Lambda$ is $(\beta_\Lambda,K)$-non resonant modulo $\Lambda$ and, for $\Lambda \neq \{0\}$, $\delta_\Lambda$-close to $\Lambda$-resonances with
\begin{equation*}
\beta_\Lambda=\frac{9MKr_\Lambda}{8}, \quad \delta_\Lambda=\frac{mr_\Lambda}{8}.
\end{equation*}
\end{proposition}

Exactly as in \cite{Pos93}, the introduction of the parameters $M$ and $m$ in the above statement is unnecessary (the above proposition does not depend on the assumptions~\eqref{bound} and~\eqref{convex}); these parameters are just here for later convenience. 

\begin{proof}
Recall that~\eqref{decomp1} gives the decomposition
\[\R^n=B_{\{0\}} \cup B_1 \cup \dots B_{n-1} \cup B_{n}\]
that can also be written as
\[ \R^n=\bigcup_{\Lambda \in M_K^a}B_\Lambda. \]
We now define those resonant blocks by choosing the parameters
\[ \lambda_{d}:=\frac{m\gamma}{8F^{n-d+1}(n+1)^\tau K^{(n+1)\tau+n-d+1}}, \quad F:=10M/m \quad 1 \leq d \leq n. \]
With this choice, the inequalities~\eqref{lambda} are satisfied (since $m\leq 1$), hence Lemma~\ref{nonresonant} can be applied with $E:=9M/m \leq F-1=10M/m-1$ (since $M/m\geq 1$). By definition, for each non-trivial $\Lambda$ of rank $d$, the block $B_{\Lambda}$ is $\delta_\Lambda$-close to $\Lambda$-resonances with
\[ \delta_\Lambda=\frac{\lambda_d}{|\tilde{\Lambda}|}=\frac{m\gamma}{8F^{n-d+1}|\tilde{\Lambda}|(n+1)^\tau K^{(n+1)\tau+n-d+1}}=\frac{mr_\Lambda}{8}. \]
Moreover, using~\eqref{beta} from Lemma~\ref{nonresonant}, $B_\Lambda$ is $(\beta_\Lambda,K)$-non resonant modulo $\Lambda$ with
\[ \beta_\Lambda=EK\delta_\Lambda=9Mm^{-1}K\frac{mr_\Lambda}{8}=\frac{9MKr_\Lambda}{8}, \]
while, for $\Lambda=\{0\}$, $B_{\{0\}}$ is $(\beta_{\{0\}},K)$-non resonant with
\[ \beta_{\{0\}} \geq \lambda_1=\frac{m\gamma}{8F^{n}(n+1)^\tau K^{(n+1)\tau+n}} \geq \frac{9MK\gamma}{8F^{n+1}(n+1)^\tau K^{(n+1)\tau+n+1}}=\frac{9MKr_{\{0\}}}{8} \]
as $F \geq 9M/m$. If we define
\[ D_{\Lambda}:=\{I \in D \; | \; \nabla h(I) \in B_{\Lambda}\}, \quad \Lambda \in M_K^a, \]
this defines a covering (up to removing such sets which are empty) of $D$ with all the required properties, and this concludes the proof.
\end{proof}

\section{Normal form and stability estimates}\label{s3}

\subsection{Normal form}\label{s31}

Let us come back to our original Hamiltonian~\eqref{Ham}, and recall that the extended integrable Hamiltonian is given by
\[ \bar{h}(I,J)=h(I)+\alpha\cdot J, \quad (I,J)\in \bar{D}=D\times \R^n. \]
Let us fix $\Lambda \in M_K^a$. Quite obviously, if a subset $D_* \subset D$ is $(\beta,K)$-non resonant modulo $\Lambda$ for $h$, then $\bar{D}_*:=D_* \times \R^m \subset \bar{D}$ is $(\beta,K)$-non resonant modulo $\Lambda$ for $\bar{h}$.

Now a Hamiltonian of the form
\[ \bar{h}(I,J)+g(\theta,\varphi,I,J) \]
is said to be in $\Lambda$-resonant normal form if
\[ g(\theta,\varphi,I,J)=\sum_{(k,l) \in \Lambda}g_{k,l}(I,J)e^{i (k,l)\cdot(\theta,\varphi)}.  \]
Such Hamiltonians have additional first integrals: indeed, given any vector $(\omega,\alpha)$ which belongs to the real subspace orthogonal to $\Lambda$, and any solution $(\theta(t),\varphi(t),I(t),J(t))$ of the system associated to $\bar{h}+g$, we have
\begin{eqnarray*}
(\omega,\alpha)\cdot (\dot{I}(t),\dot{J}(t)) & = & -(\omega,\alpha)\cdot\partial_{(\theta,\varphi)}(\bar{h}(I(t),J(t))+g(\theta(t),\varphi(t),I(t),J(t))) \\
& = & -(\omega,\alpha)\cdot\partial_{(\theta,\varphi)}(g(\theta(t),\varphi(t),I(t),J(t))) \\
& = & -(\omega,\alpha)\cdot \sum_{(k,l) \in \Lambda}i2\pi(k,l)g_{k,l}(I(t),J(t))e^{i (k,l)\cdot(\theta(t),\varphi(t))} \\
& = & -i\sum_{(k,l) \in \Lambda}(\omega,\alpha)\cdot(k,l)g_{k,l}(I(t),J(t))e^{i (k,l)\cdot(\theta(t),\varphi(t))}\\
& = & 0.
\end{eqnarray*}
In the special case where $\Lambda=\{0\}$, it is straightforward to see that $g$ is in fact independent of the angles $(\theta,\varphi)$, so that the resonant normal form is integrable.

We can now state the normal form lemma, which states that on a sufficiently small neighborhood of a non-resonant domain modulo $\Lambda$, up to a real-analytic symplectic transformation which is close to the identity, the original Hamiltonian can be written as a $\Lambda$-resonant normal form $\bar{h}+g$ up to an exponentially small remainder.

\begin{lemma}\label{normal}
Let $H$ be as in~\eqref{Ham}, with $h$ satisfying~\eqref{bound} and f satisfying~\eqref{pert}. Consider a domain $\bar{D}_*=D_* \times \R^m \subset \bar{D}$ which is $(\beta,K)$-non resonant modulo $\Lambda$ for $\bar{h}$, and given some parameter $r>0$, assume that
\begin{equation}\label{seuil1}
\varepsilon \leq \frac{\beta r}{2^79K}, \quad r \leq \frac{8\beta}{9MK}, \quad r \leq r_0,
\end{equation}
and $Ks_0 \geq 6$. Then there exists a real-analytic symplectic embedding 
\[ \Phi: V_{\tilde{r}}\bar{D}_* \times V_{\tilde{s}_0}\T^{n+m} \rightarrow V_{r}\bar{D}_* \times V_{s_0}\T^{n+m}, \quad \tilde{r}:=r/2, \quad \tilde{s}_0:=s_0/6,\]
such that
\[ H \circ \Phi=\bar{h}+g+f_* \]
where $\bar{h}+g$ is in $\Lambda$-resonant normal form with the estimates
\begin{equation}\label{estrest}
|g+f_*|_{\tilde{r},\tilde{s}_0} \leq 2\varepsilon, \quad |f_*|_{\tilde{r},\tilde{s}_0} \leq e^{-Ks_0/6}\varepsilon 
\end{equation} 
and 
\begin{equation}\label{estdist}
\sup_{(I,J,\theta,\varphi)\in V_{\tilde{r}}\bar{D}_* \times V_{\tilde{s}_0}\T^{n+m}}||\Pi_{I,J}\Phi(I,J,\theta,\varphi)-(I,J)||\leq \frac{18K\varepsilon}{\beta}
\end{equation}
where $\Pi_{I,J}$ denotes the projection onto the action space coordinates.
\end{lemma}

This statement is a direct consequence of the Normal Form Lemma of~\cite{Pos93} (with the choice of the constants $p=9/8$ and $q=9$), to which we refer for a proof. More detailed estimates on $g$ and on $\Psi$ are available, but they will not be needed.

\subsection{Non-resonant stability estimates}\label{s32}

In the special case where the domain $\bar{D}_*$ is $(\beta,K)$-non resonant modulo $\Lambda$ for $\bar{h}$, with $\Lambda=\{0\}$, the normal form obtained in the previous section is, as we already said, integrable up to an exponentially small remainder. It is very easy to prove that in this case the action variables $(I(t),J(t))$ remain stable for an exponentially long interval of time, and this does not require any convexity assumptions on $h$. Here's a precise statement.

\begin{proposition}\label{propstab1}
Let $H$ be as in~\eqref{Ham}, with $h$ satisfying~\eqref{bound} and f satisfying~\eqref{pert}. Consider a domain $\bar{D}_*=D_* \times \R^m \subset \bar{D}$ which is $(\beta,K)$-non resonant modulo $\Lambda=\{0\}$ for $\bar{h}$, and given some parameter $r>0$, assume that~\eqref{seuil1} is satisfied.
Then for every solution with initial action $(I_0,J_0) \in \bar{D}_* $ we have
\[ ||(I(t),J(t))-(I_0,J_0)|| \leq r, \quad |t|\leq \frac{s_0r}{5\varepsilon}e^{Ks_0/6}. \]
\end{proposition} 

This is the content of Proposition $1$ (Nonresonant stability estimate) of~\cite{Pos93}, to which we refer once again for a proof.  

\subsection{Resonant stability estimates}\label{S33}

Next we study the case where the domain $\bar{D}_*$ is $(\beta,K)$-non resonant modulo $\Lambda$ for $\bar{h}$, with $\Lambda$ non-trivial. The domain $\bar{D}_*=D_* \times \R^m$ will be said to be $\delta$-close to $\Lambda$-resonances if $D_*$ is $\delta$-close to $\Lambda$-resonances, as defined previously.

Assuming convexity of $\bar{h}$, one knows how to bound the variation of the action $(I(t),J(t))$ using conservation of the energy and convexity arguments, as was first proved in \cite{BG86} and later in~\cite{Loc92}. However, in our situation $\bar{h}$ is not convex, but it is convex with respect to the $I$ variables and linear in the $J$ variables, and we will prove below that this is sufficient to bound the variation of the action variables $I(t)$ only, loosing control on the evolution of $J(t)$.   

\begin{proposition}\label{propstab2}
Let $H$ be as in~\eqref{Ham}, with $h$ satisfying~\eqref{bound},~\eqref{bound2} and~\eqref{convex} and f satisfying~\eqref{pert}. Consider a domain $\bar{D}_*=D_* \times \R^m \subset \bar{D}$ which is $(\beta,K)$-non resonant modulo $\Lambda$ for $\bar{h}$, with $\Lambda\neq \{0\}$, but $\delta$-close to $\Lambda$-resonances. Given some parameter $r>0$, assume that
\begin{equation}\label{seuil2}
\varepsilon \leq \frac{mr^2}{2^{10}}, \quad \delta = \frac{mr}{8}, \quad r \leq \frac{8\beta}{9MK}, \quad r \leq r_0.
\end{equation}
Then for every solution with initial action $(I_0,J_0) \in \bar{D}_* $ we have
\[ ||I(t)-I_0|| \leq r, \quad |t|\leq \frac{m s_0r^2}{288\Omega\varepsilon}e^{Ks_0/6}. \]
\end{proposition} 

\begin{proof}
First we assume that $Ks_0 \geq 6$. Since $m \leq M$, one easily check that~\eqref{seuil2} implies~\eqref{seuil1}, and therefore Lemma~\ref{normal} can be applied: there exists a real-analytic symplectic embedding 
\[ \Phi: V_{\tilde{r},\tilde{s}_0}\bar{D}_* \rightarrow V_{r,s_0}\bar{D}_*\]
such that
\[ H \circ \Phi=\bar{h}+g+f_* \]
where $\bar{h}+g$ is in $\Lambda$-resonant normal form with the estimates~\eqref{estrest} and~\eqref{estdist}. Now, from~\eqref{estdist} and~\eqref{seuil2} we get
\[ ||\Pi_{I,J}\Phi-\mathrm{Id}||\leq \frac{18K\varepsilon}{\beta} \leq \frac{2^4\varepsilon}{Mr}\leq \frac{mr}{2^6M}=\frac{\delta}{8M}. \]
Therefore the inverse image of $\bar{D}_*\times \T^{n+m}$ by $\Phi$ is contained in $U_\rho\bar{D}_*\times \T^{n+m}$, where $\rho:=\delta/(4M)$. Recall that $\tilde{r}=r/2$ and $\tilde{s}_0=s_0/6$. We claim that for any initial action $(\tilde{I}_0,\tilde{J}_0) \in U_\rho\bar{D}_*$, the solution $(\tilde{I}(t),\tilde{J}(t),\tilde{\theta}(t),\tilde{\varphi}(t))$ of the system associated to the Hamiltonian $H \circ \Phi$ satisfies
\[ ||\tilde{I}(t)-\tilde{I}_0||\leq \tilde{r}-\rho, \quad |t|\leq \frac{m\tilde{s}_0\tilde{r}^2}{12\Omega\varepsilon}e^{Ks/6}. \]
Assuming this claim, for any initial action $(I_0,J_0) \in \bar{D}_*$, the solution $(I(t),J(t),\theta(t),\varphi(t))$ of the system associated to the Hamiltonian $H$ satisfies
\[ ||I(t)-I_0||\leq  ||I(t)-\tilde{I}(t)||+||\tilde{I}(t)-\tilde{I}_0||+||\tilde{I}_0-I_0||  \leq  \rho/2+(\tilde{r}-\rho)+\rho/2 \leq \tilde{r} \leq r \]
for times
\[ |t|\leq \frac{ms_0r^2}{288\Omega\varepsilon}e^{Ks_0/6}  \]
which is exactly the statement we want to prove. It is therefore sufficient to prove the above claim.

To simplify notations, let us drop the tildes and simply write $(I_0,J_0) \in U_\rho\bar{D}_*$ and $(I(t),J(t),\theta(t),\varphi(t))$ the associated solution. Let $B$ be the ball of radius $\tilde{r}-\rho$ around $I_0$, then $B \times \R^m$ is contained in $U_{\tilde{r}}\bar{D}_*$. Let $T_e$ be the positive time (possibly infinite) of first exit of $(I(t),J(t))$ from $B \times \R^m$: it is then also the time of first exit of $I(t)$ from $B$. Let also
\[ T_*:=\frac{m\tilde{s}_0\tilde{r}^2}{12\Omega\varepsilon}e^{Ks_0/6}, \quad T:=\min\{T_e,T_*\}. \]
Furthermore, let us write
\[ \Delta\bar{h}:=\bar{h}(I(T),J(T))-\bar{h}(I_0,J_0) \in \R, \]
\[ \Delta(I,J):=(I(T),J(T))-(I_0,J_0)=(I(T)-I_0,J(T)-J_0) \in \R^{n+m}, \]
\[ \Delta I:=I(T)-I_0 \in \R^{n}, \]
\[ I(s):=I_0+s\Delta I \in \R^{n}, \quad 0 \leq s \leq 1. \]  
By definition of $\bar{h}$, we have $\nabla \bar{h}(I,J)=(\nabla h(I),\alpha) \in \R^{n+m}$ and 
\begin{equation*}
\nabla^2 \bar{h}(I,J)=
\begin{pmatrix}
\nabla^2 h(I) & 0 \\
0 & 0 
\end{pmatrix}
\in M_{n+m}(\R), \quad \nabla^2 h(I) \in M_{n}(\R).
\end{equation*}
Using Taylor's formula with integral remainder at the point $(I_0,J_0)$ and the special form of $\nabla^2 \bar{h}$, we get, letting $\omega:=\nabla h(I_0)$,
\[ \Delta\bar{h}=(\omega,\alpha)\cdot\Delta(I,J)+\int_0^1 (1-s)\nabla^2 h(I(s)) \Delta I \cdot \Delta I ds. \]
Using the assumption~\eqref{convex}, that is the convexity of $h$, we obtain
\begin{equation}\label{conf1}
|\Delta \bar{h}| + |(\omega,\alpha)\cdot\Delta(I,J)| \geq m/2||\Delta I||^2.
\end{equation}
By conservation of energy $H$ and using the first part of~\eqref{estrest} together with the first part of~\eqref{seuil2} and the definition of $\tilde{r}$, we have
\begin{equation}\label{conf2}
|\Delta \bar{h}| \leq 2|g+f_*|_{\tilde{r},\tilde{s}_0} \leq 4\varepsilon \leq \frac{mr^2}{2^8}=\frac{m\tilde{r}^2}{2^6}.  
\end{equation}
Then, by definition, $(I_0,J_0) \in U_\rho\bar{D}_*$, so there exist $(I_0',J_0') \in \bar{D}_*$ which is $\rho$-close to $(I_0,J_0)$, but then by assumption $\bar{D}_*$ is $\delta$-close to $\Lambda$-resonances, hence there exists $\omega_* \in \R^n$ such that the vector $(\omega_*,\alpha)$ belongs to the real subspace orthogonal to $\Lambda$ and is $\delta$-close to $(\nabla h(I_0'),\alpha)$. Therefore, since $h$ satisfies~\eqref{bound} and by definition of $\rho$, we obtain
\begin{equation}\label{oubli}
||\omega-\omega_*||=||\nabla h(I_0)-\omega_*||=||\nabla h(I_0)-\nabla h(I_0')||+ ||\nabla h(I_0')-\omega_*|| \leq M\rho+\delta=5\delta/4. 
\end{equation}
Writing 
\begin{eqnarray*}
(\omega,\alpha)\cdot\Delta(I,J) & = & (\omega,\alpha)\cdot\Delta(I,J)-(\omega_*,\alpha)\cdot\Delta(I,J)+(\omega_*,\alpha)\cdot\Delta(I,J) \\
& = & (\omega-\omega_*,0)\cdot\Delta(I,J)+(\omega_*,\alpha)\cdot\Delta(I,J) \\
& = & (\omega-\omega_*)\cdot\Delta I+(\omega_*,\alpha)\cdot\Delta(I,J)
\end{eqnarray*}
we obtain
\begin{equation}\label{conf3}
|(\omega,\alpha)\cdot\Delta(I,J)| \leq |(\omega-\omega_*)\cdot\Delta I|+|(\omega_*,\alpha)\cdot\Delta(I,J)|.
\end{equation}
Using~\eqref{oubli}, the second part of~\eqref{seuil2} and the definition of $\tilde{r}$, we can bound the first summand by
\begin{equation}\label{conf4}
|(\omega-\omega_*)\cdot\Delta I|\leq ||\omega-\omega_*||||\Delta I||\leq (5\delta/4)||\Delta I|| \leq \frac{5\delta^2}{2m}+\frac{m}{6}||\Delta I||^2 = \frac{5m\tilde{r}^2}{32}+ \frac{m}{6}||\Delta I||^2.
\end{equation}
For the second summand, using the fact that $(\omega_*,\alpha)$ belongs to the real subspace orthogonal to $\Lambda$, and that $\bar{h}+g$ is in $\Lambda$-resonant normal form, we have
\begin{eqnarray*}
|(\omega_*,\alpha)\cdot\Delta(I,J)| & \leq & \int_0^T |(\omega_*,\alpha)\cdot \partial_{\theta,\varphi}f_*(I(t),J(t),\theta(t),\varphi(t))|dt \\
& \leq & T||(\omega_*,\alpha)||\sup_{(I,J,\theta,\varphi)\in B\times \R^n \times \T^{n+m}} ||\partial_{\theta,\varphi}f_*(I,J,\theta,\varphi)||.
\end{eqnarray*}
Now using a Cauchy estimate and the second part of~\eqref{estrest} we get
\[ ||\partial_{\theta,\varphi}f_*(I,J,\theta,\varphi)|| \leq \frac{1}{e\tilde{s}_0}|f_*|_{\tilde{r},\tilde{s}} \leq \frac{\varepsilon}{e\tilde{s}_0}e^{-Ks_0/6}. \]
Moreover, as $T \leq T_*$ and $||(\omega_*,\alpha)||\leq \Omega$ since $h$ satisfies~\eqref{bound2}, we get
\begin{equation}\label{conf5}
|(\omega_*,\alpha)\cdot\Delta(I,J)| \leq T_*\Omega\frac{\varepsilon}{e\tilde{s}_0}e^{-Ks_0/6}=\frac{m\tilde{r}^2}{12e}.
\end{equation}
Putting together~\eqref{conf1},~\eqref{conf2},~\eqref{conf3},~\eqref{conf4} and~\eqref{conf5} we eventually arrive at
\[ m/3||\Delta I||^2 \leq \frac{m\tilde{r}^2}{64}+\frac{5m\tilde{r}^2}{32}+\frac{m\tilde{r}^2}{12e}\leq \frac{m\tilde{r}^2}{4}. \]
Now $\rho=\delta/(4M)=mr/(32M)=m\tilde{r}/(16M)\leq \tilde{r}/16$ as $m/M\leq 1$, thus $\tilde{r}-\rho\geq 15\tilde{r}/16$ and in particular $\tilde{r}^2<4(\tilde{r}-\rho)^2/3$. This, together with the last inequality, implies that 
\[ m/3||\Delta I||^2 < m/3 (\tilde{r}-\rho)^2\]
and therefore $||\Delta I|| < \tilde{r}-\rho$. This eventually proves that $T=T_*$, that is
 \[ ||\tilde{I}(t)-\tilde{I}_0||\leq \tilde{r}-\rho, \quad 0\leq t\leq T_*=\frac{m\tilde{s}_0\tilde{r}^2}{12\Omega\varepsilon}e^{Ks/6}. \]
The same argument yields the same result for negative times $-T_* \leq t \leq 0$, and hence
\[ ||\tilde{I}(t)-\tilde{I}_0||\leq \tilde{r}-\rho, \quad |t|\leq \frac{m\tilde{s}_0\tilde{r}^2}{12\Omega\varepsilon}e^{Ks/6} \]
which was the claim that needed to be proved. This ends the proof under the assumption $Ks_0\geq 6$. But if $Ks_0<6$, the exact same argument applies to the original Hamiltonian $H$, by setting $g=0$ and $f_*=f$, and this concludes the proof. 
\end{proof}

\section{Proof of the main result}\label{s4}

This section is devoted to the proof of our main result, Theorem~\ref{thm1}, which will be easily obtained using Proposition~\ref{covering}, Proposition~\ref{propstab1} and Proposition~\ref{propstab2}.

\begin{proof}[Proof of Theorem~\ref{thm1}]
Recall that we are considering $H$ as in~\eqref{Ham}, with $h$ satisfying~\eqref{bound},~\eqref{bound2} and~\eqref{convex}, $f$ satisfying~\eqref{pert}, and $\alpha$ satisfying~\eqref{dio}. Recall also that we have defined the positive constants $a$ and $b$ by
\[ a=\frac{1}{2(n+1)(\tau+1)}, \quad b=\frac{(n+1)\tau+1}{2(n+1)(\tau+1)}. \]
We can apply Proposition~\ref{covering} to obtain a covering of $\bar{D}=D\times \R^m$ by subsets $\bar{D}_{\Lambda}=D_\Lambda \times \R^m$, where $\Lambda \in M_K^a$, such that each $\bar{D}_{\Lambda}$ is $(\beta_\Lambda,K)$-non resonant modulo $\Lambda$ for $\bar{h}$ and, for $\Lambda \neq \{0\}$, $\delta_\Lambda$-close to $\Lambda$-resonances for $\bar{h}$ with 
\begin{equation}\label{rayon2}
\beta_\Lambda=\frac{9MKr_\Lambda}{8}, \quad \delta_\Lambda=\frac{mr_\Lambda}{8},
\end{equation}
where 
\begin{equation}\label{rayon}
r_\Lambda:=\frac{\gamma}{|\tilde{\Lambda}|F^{n-d+1}(n+1)^\tau K^{(n+1)\tau+n-d+1}}, \quad d=\mathrm{rank}\:\Lambda.
\end{equation}
Then, given any $\Lambda \in M_K^a$, we can apply either Proposition~\ref{propstab1} (for $\Lambda=\{0\}$) of Proposition~\ref{propstab2} (for $\Lambda\neq \{0\}$), with $r=r_\Lambda$, $\beta=\beta_\Lambda$ and $\delta=\delta_\Lambda$ provided that~\eqref{seuil1} and~\eqref{seuil2} are satisfied. The inequalities~\eqref{seuil1} are easily seen to be implied by~\eqref{seuil2} as we already pointed out, hence we only need to verify~\eqref{seuil2}, and in view of our definitions of $r_\Lambda$, $\beta_\Lambda$ and $\delta_\Lambda$, the latter reduces to  
\begin{equation}\label{seuil3}
\varepsilon \leq \frac{mr_\Lambda^2}{2^{10}}, \quad r_\Lambda \leq r_0.
\end{equation}
Given any $\Lambda \in M_K^a$, we have
\[ \frac{\gamma}{F^{n+1}(n+1)^\tau K^{(n+1)(\tau+1)}} \leq r_\Lambda \leq \frac{\gamma}{F(n+1)^\tau K^{(n+1)\tau+1}},  \] 
as $F \geq 1$ and $K \geq 1$. Hence~\eqref{seuil3} is satisfied, for any $\Lambda \in M_K^a$, if
\begin{equation}\label{seuil4}
\varepsilon \leq \frac{m\gamma^2}{2^{10}F^{2(n+1)}(n+1)^{2\tau}K^{2(n+1)(\tau+1)}}, \quad K\geq \left(\frac{\gamma}{F(n+1)^\tau r_0}\right)^{\frac{a}{b}}.
\end{equation}
Let us define
\[ \varepsilon_0:=\frac{m\gamma^2}{2^{10}F^{2(n+1)}(n+1)^{2\tau}}, \quad \varepsilon_*:=\varepsilon_0\left(\frac{F(n+1)^\tau r_0}{\gamma}\right)^{\frac{1}{b}}  \]
and 
\[ \quad K:=\left(\frac{\varepsilon_0}{\varepsilon}\right)^a, \]
then~\eqref{seuil4} eventually reduces to
\begin{equation}\label{seuil5}
\varepsilon \leq \min\{\varepsilon_0,\varepsilon_*\}.
\end{equation}
Under this choice of $K$ and this smallness assumption on $\varepsilon$, Proposition~\ref{propstab1} and Proposition~\ref{propstab2} apply and in the resonant case, we obtain 
\[ ||I(t)-I_0|| \leq \frac{\gamma}{F(n+1)^\tau K^{(n+1)\tau+1}} \leq \frac{\gamma}{F(n+1)^\tau}\left(\frac{\varepsilon}{\varepsilon_0}\right)^b \]
up to times
\[ |t|\leq T_{\Lambda}\exp\left(\frac{s_0}{6}\left(\frac{\varepsilon_0}{\varepsilon}\right)^a\right), \quad T_{\Lambda}:=\frac{m s_0r_{\Lambda}^2}{288\Omega\varepsilon},\]
while in the non-resonant case, we obtain
\[ ||I(t)-I_0|| \leq ||(I(t),J(t))-(I_0,J_0)||\leq  \frac{\gamma}{F(n+1)^\tau K^{(n+1)\tau+1}} \leq \frac{\gamma}{F(n+1)^\tau}\left(\frac{\varepsilon}{\varepsilon_0}\right)^b  \]
up to times
\[ \quad |t|\leq T_{\{0\}}\exp\left(\frac{s_0}{6}\left(\frac{\varepsilon_0}{\varepsilon}\right)^a\right), \quad T_{\{0\}}:=\frac{s_0r_{\{0\}}}{5\varepsilon}. \]
To obtain a uniform time estimates, observe that for any $\Lambda \in M_K^a$ we have
\[ T_{\Lambda}\geq \frac{2^{10}}{288}\frac{s_0}{\Omega} \geq \frac{3s_0}{\Omega}.  \]
Letting 
\[ R_*=\frac{\gamma}{F(n+1)^\tau}, \quad T_*=\frac{3s_0}{\Omega}, \]
we have just proved that
\[ ||I(t)-I_0|| \leq R_*\left(\frac{\varepsilon}{\varepsilon_0}\right)^b, \quad |t|\leq T_*\exp\left(\frac{s_0}{6}\left(\frac{\varepsilon_0}{\varepsilon}\right)^a\right), \]
provided~\eqref{seuil5} is satisfied. Recalling that $F=10M/m$, this was exactly the statement to be proved.
\end{proof}

\section{Improved stability close to resonances}\label{s5}

For solutions starting close to resonances, we can obtain a better result. Consider a fixed submodule $L$ of $\Z^{n+m}$ of rank $d$, which is assumed to be admissible and maximal, and let $K_{L} \geq 1$ such that $L$ is a $K_L$-submodule. Recall that $\tilde{L}$ denotes the projection of $L$ onto $\R^n$, and $|\tilde{L}|$ denotes the co-volume of $\tilde{L}$. To such a $L$ the space of $L$-resonances is 
\[ R_L=\{\omega \in \R^n \; | \; (k,l)\cdot (\omega,\alpha)=0, \; \forall (k,l) \in L\}\]
and let
\[ S_L:=\{ (I,J) \in \bar{D} \; | \; \nabla h(I) \in R_L\}=\{ I \in D \; | \; \nabla h(I) \in R_L\} \times \R^m \]
the resonant domain in action space. Solutions with initial action close to $S_L$ satisfy better stability estimates, as stated in the theorem below.

\begin{theorem}\label{thm2}
Let $H$ be as in~\eqref{Ham}, with $h$ satisfying~\eqref{bound},~\eqref{bound2} and~\eqref{convex}, and $f$ satisfying~\eqref{pert}. Assume also that $\alpha$ satisfies~\eqref{dio}, and let us define
\[ a(d)=\frac{1}{2(n+1)\tau+n+1-d}, \quad b(d)=\frac{(n+1)\tau+1}{2(n+1)\tau+n+1-d}, \] 
\[ R_*(d)=\left(\frac{m}{10M}\right)^{n+1-d}\frac{\gamma}{(n+1)^\tau}, \]
and, given a submodule $L$ as above, let
\[ \varepsilon_0(L)=\frac{1}{|\tilde{L}|}\frac{m\gamma^2}{2^{10}(n+1)^{2\tau}}\left(\frac{m}{10M}\right)^{2(n+1-d)}, \quad \varepsilon_*(L)=\varepsilon_0(L)\left(\frac{r_0}{R_*(d)}\right)^{\frac{1}{b(d)}} \]
\[ \varepsilon_{**}(L):=\varepsilon_0(L)K_L^{-\frac{1}{a(d)}}. \] 
If $\varepsilon \leq \min\{\varepsilon_0(L),\varepsilon_*(L),\varepsilon_{**}(L)\}
$, for any solution $(I(t),J(t),\theta(t),\varphi(t))$ of the system associated to $H$ with initial condition $(I_0,J_0,\theta_0,\varphi_0) \in U_{\rho}S_L \times \T^{n+m}$, with $\rho=4M^{-1}\sqrt{m\varepsilon}$, we have
\[ ||I(t)-I_0||\leq R_*(d)\left(\frac{\varepsilon}{\varepsilon_0(L)}\right)^{b(d)}, \quad |t|\leq T_*\exp\left(\frac{s_0}{6}\left(\frac{\varepsilon_0(L)}{\varepsilon}\right)^{a(d)}\right) \]
with $T_*$ as in Theorem~\ref{thm1}.
\end{theorem}

Observe that in the special case where $L=\{0\}$, $S_L=D\times \R^m$, $|\tilde{L}|=K_L=1$ and therefore the above statement exactly reduces to Theorem~\ref{thm1}.

\begin{proof}
Consider the subset 
\[ M_K^a(L):=\{\Lambda \in M_K^a \; | \; \Lambda \supseteq L\}\]
and assume, for any $\Lambda \in M_K^a(L)$,
\begin{equation}\label{seu1}
\varepsilon \leq \frac{mr_\Lambda^2}{2^{10}}, \quad  r_\Lambda \leq r_0, \quad K \geq K_L
\end{equation}
where the last inequality is to ensure that $M_K^a(L)$ is actually non empty.

The resonant blocks $B_\Lambda$, for $\Lambda \in M_K^a(L)$, cover the resonant zone $Z_L$ and hence their pullbacks cover $U_{\rho'} S_L$, with
\[ \rho'=\frac{\delta_L}{M}=\frac{mr_L}{8M} \geq \frac{\sqrt{2^{10}}}{8M}\sqrt{m\varepsilon}=4M^{-1}\sqrt{m\varepsilon}=\rho. \]
Moreover, for any $\Lambda \in M_K^a(L)$, we have
\[ \frac{\gamma}{|\tilde{L}|F^{n-d+1}(n+1)^\tau K^{(n+1)\tau+n-d+1}}=r_L \leq r_\Lambda \leq \frac{\gamma}{F(n+1)^\tau K^{(n+1)\tau+1}}.\] 
Using the above inequalities and proceeding exactly as in the proof of Theorem~\ref{thm1}, we can define
\[ K:=\left(\frac{\varepsilon_0(L)}{\varepsilon}\right)^{a(d)} \]
and verify that~\eqref{seu1} is implied by
\[ \varepsilon \leq \varepsilon_0(L), \quad \varepsilon \leq \varepsilon_*(L), \quad \varepsilon \leq \varepsilon_{**}(L). \]
The stability estimates apply uniformly to all blocks $B_\Lambda$, for $\Lambda \in M_K^a(L)$, and one easily check that the statement follows with the given constants. 
\end{proof}

\section{Improved stability far away from resonances}\label{s6}

We now investigate solutions which start far away from resonances. In this special case, we will actually be able to control the variation not only of the $I$ but also of the $J$ variables.

Results of this section do not depend on $h$ being convex or its gradient being bounded, that is~\eqref{convex} and~\eqref{bound2} are unnecessary. It will be sufficient to assume the existence of $\bar{m}>0$ such that for any Lebesgue measurable subset $\Omega \in \R^n$, we have the measure estimate
\begin{equation}\label{bound3}\tag{$\bar{m}$}
\mathrm{Leb}(\nabla h^{-1}(\Omega) \cap D) \leq \bar{m}^{-1}\mathrm{Leb}(\Omega)
\end{equation} 
where $\mathrm{Leb}$ denotes the Lebesgue measure on $\R^n$. Certainly, the convexity assumption~\eqref{convex} implies~\eqref{bound3} with $\bar{m}=m$, but the latter is more general: in particular, it holds true if $h$ is Kolmogorov non-degenerate, that is if the determinant of the Hessian $\nabla h^2(I)$ is uniformly bounded away from zero for any $I \in D$.

Now consider $\Lambda \in M_{K,1}^a$ an admissible maximal $K$-submodule of $\Z^{n+m}$ of rank $1$. Then $\Lambda$ contains a unique vector $(k,l)\in \Z^{n} \setminus \{0\} \times \Z^m$ such that $|(k,l)| \leq K$ and such that its components are relatively primes. The submodule $\tilde{\Lambda}$ of $\Z^n$ is then generated by $k$, and $|\tilde{\Lambda}|=||k||$. In the sequel, we shall write $\Lambda=\Lambda(k,l)$.

With these notations, we recall that the completely non-resonant block $B_{\{0\}} \subset \R^n$, introduced in Section~\S\ref{s2}, can be defined by
\[ B_{\{0\}}:=\left\{\omega \in \R^n \; | \; ||\omega-R_{\Lambda(k,l)}||\geq \frac{\lambda_1}{||k||} \right\}, \quad  \lambda_1=\frac{\gamma}{8F^n(n+1)^\tau K^{(n+1)\tau+n}}, \quad F=10M/\bar{m}. \]
Therefore, from Proposition~\ref{nonresonant2}, this set is $(\beta_{\{0\}},K)$-non resonant with 
\[ \beta_{\{0\}}=\frac{9MKr_{\{0\}}}{8}, \quad r_{\{0\}}=\frac{\gamma}{F^{n+1}(n+1)^\tau K^{(n+1)(\tau+1)}}.  \]
If we further define
\[ D_{\{0\}}:=\{ I \in D \; | \; \nabla h(I) \in B_{\{0\}} \}, \quad \bar{D}_{\{0\}}:=D_{\{0\}} \times \R^m, \]
we arrive at the following statement.

\begin{theorem}\label{thm3}
Let $H$ be as in~\eqref{Ham}, with $h$ satisfying~\eqref{bound} and~\eqref{bound3}, and $f$ satisfying~\eqref{pert}. Assume also that $\alpha$ satisfies~\eqref{dio}, and let us define
\[ \bar{R}_*:=\left(\frac{\bar{m}}{10M}\right)^{n+1}\frac{\gamma}{(n+1)^\tau}, \quad \bar{\varepsilon}_0:=\frac{M\gamma^2}{2^{10}(n+1)^{2\tau}}\left(\frac{\bar{m}}{10M}\right)^{2(n+1)}, \quad \bar{\varepsilon}_*=\bar{\varepsilon}_0\left(\frac{r_0}{\bar{R}_*}\right)^2.\] 
If $\varepsilon \leq \min\{\bar{\varepsilon}_0,\bar{\varepsilon}_*\}$, for any solution $(I(t),J(t),\theta(t),\varphi(t))$ of the system associated to $H$ with initial condition $(I_0,J_0,\theta_0,\varphi_0) \in \bar{D}_{\{0\}} \times \T^{n+m}$, we have
\[ ||(I(t),J(t))-(I_0,J_0)||\leq \bar{R}_*\left(\frac{\varepsilon}{\bar{\varepsilon}_0}\right)^{\frac{1}{2}}, \quad |t|\leq \frac{s_0\bar{R}_*}{5\bar{\varepsilon}_0}\left(\frac{\bar{\varepsilon}_0}{\varepsilon}\right)^{\frac{1}{2}}\exp\left(\frac{s_0}{6}\left(\frac{\bar{\varepsilon}_0}{\varepsilon}\right)^a\right). \]
Moreover, the complement of $D_{\{0\}}$ in $\R^n$ has a measure of order $\varepsilon^b$, with $a$ and $b$ as in Theorem~\ref{thm1}.
\end{theorem}

Observe that the measure estimate on the complement of $D_{\{0\}}$ is better than the one obtained in \cite{Pos93} for $\tau=0$, as in the latter reference it is only of order one. Observe also that it is only for this measure estimate that~\eqref{bound3} is needed; the first part of the statement holds true without this assumption.

\begin{proof}
Proposition~\ref{propstab1} can be applied with $r=r_{\{0\}}$, provided that 
\begin{equation}\label{se}
\varepsilon \leq \frac{\beta_{\{0\}}r_{\{0\}}}{2^79K}=\frac{Mr_{\{0\}}^2}{2^{10}}, \quad r_{\{0\}} \leq r_0, 
\end{equation}
holds true. Choosing 
\[ K:=\left(\frac{\bar{\varepsilon}_0}{\varepsilon}\right)^a, \]
we have
\[ r_{\{0\}}=\bar{R}_*\left(\frac{\varepsilon}{\bar{\varepsilon}_0}\right)^{\frac{1}{2}} \]
and the inequalities~\eqref{se} are satisfied if $\varepsilon \leq \min\{\bar{\varepsilon}_0,\bar{\varepsilon}_*\}$, where $\bar{\varepsilon}_0$, $\bar{\varepsilon}_*$ and $\bar{R}_*$ are the constants given in the statement. It then follows from Proposition~\ref{propstab1} that
\[ ||(I(t),J(t))-(I_0,J_0)|| \leq r_{\{0\}}, \quad |t|\leq \frac{s_0r_{\{0\}}}{5\varepsilon}e^{Ks_0/6}, \]
that is
\[ ||(I(t),J(t))-(I_0,J_0)|| \leq \bar{R}_*\left(\frac{\varepsilon}{\bar{\varepsilon}_0}\right)^{\frac{1}{2}}, \quad |t|\leq \frac{s_0\bar{R}_*}{5\bar{\varepsilon}_0}\left(\frac{\bar{\varepsilon}_0}{\varepsilon}\right)^{\frac{1}{2}}\exp\left(\frac{s_0}{6}\left(\frac{\bar{\varepsilon}_0}{\varepsilon}\right)^a\right), \]
which proves the first part of the statement.

Concerning the second part of the statement, we follow~\cite{Pos93}. First, in view of~\eqref{bound3}, it suffices to show that the complement of $B_{\{0\}}$ in $\R^n$ has a relative Lebesgue measure of order $\varepsilon^b$. But by construction, the latter set is $Z_1$, the resonant zone of multiplicity $1$ defined in Section~\S\ref{s2}, which has a relative measure of order
\[ \sum_{\Lambda(k,l) \in M_{K,1}^a}\frac{\lambda_1}{||k||} \sim \frac{1}{K^{(n+1)\tau+n}} \sum_{k \in \Z^n, \; 0 < |k|\leq K}\frac{1}{||k||} \sim \frac{1}{K^{(n+1)\tau+1}} \sim \varepsilon^{a((n+1)\tau+1)}=\varepsilon^b. \]
\end{proof}

Next we look at a different, and in some sense more natural, non-resonant set. Choose $0 < \gamma' \leq \gamma$ and $\tau'$ such that $\tau'>n+m-1$ and $\tau'\geq \tau$. We define
\[ B_{\gamma',\tau'}:=\left\{\omega \in \R^n\; | \; |(k,l)\cdot (\omega,\alpha)| \geq \gamma' (|k|+|l|)^{-\tau'}\right\} \]
and
\[ D_{\gamma',\tau'}:=\{I \in D \; | \; \nabla h(I) \in  B_{\gamma',\tau'}\}, \quad \bar{D}_{\gamma',\tau'}:=D_{\gamma',\tau'} \times \R^m. \]
Those sets are clearly $(\beta',K)$-non resonant, with $\beta':=\gamma'K^{-\tau'}$, and we obtain the following result.

\begin{theorem}\label{thm33}
Let $H$ be as in~\eqref{Ham}, with $h$ satisfying~\eqref{bound} and~\eqref{bound3}, and $f$ satisfying~\eqref{pert}. Assume also that $\alpha$ satisfies~\eqref{dio}, and let us define
\[ a':=\frac{1}{2(\tau'+1)}, \quad R_*':=\frac{8\gamma'}{9M}, \quad \varepsilon_0':=\frac{\gamma'^2}{768M}, \quad \varepsilon_*'=\varepsilon_0'\left(\frac{r_0}{R_*'}\right)^2.\] 
If $\varepsilon \leq \min\{\varepsilon_0',\varepsilon_*'\}
$, for any solution $(I(t),J(t),\theta(t),\varphi(t))$ of the system associated to $H$ with initial condition $(I_0,J_0,\theta_0,\varphi_0) \in \bar{D}_{\gamma',\tau'} \times \T^{n+m}$, we have
\[ ||(I(t),J(t))-(I_0,J_0)||\leq R_*'\left(\frac{\varepsilon}{\varepsilon_0'}\right)^{\frac{1}{2}}, \quad |t|\leq \frac{s_0R_*'}{5\varepsilon_0'}\left(\frac{\varepsilon_0'}{\varepsilon}\right)^{\frac{1}{2}}\exp\left(\frac{s_0}{6}\left(\frac{\varepsilon_0'}{\varepsilon}\right)^{a'}\right). \]
Moreover, the complement of $D_{\gamma',\tau'}$ in $\R^n$ has a measure of order $\gamma'$.
\end{theorem}

As we already said, when $m \geq 2$, the set of vectors $\alpha \in \R^m$ that satisfy~\eqref{dio} with $\tau=m-1$ has zero measure, and when $\tau >m-1$, we can choose $\tau'=n+\tau$ and thus
\[  a'=\frac{1}{2(n+\tau+1)}\]
in the statement above.

\begin{proof}
Let us define
\[ r':=\frac{8\beta'}{9MK}=\frac{8\gamma'}{9MK^{\tau'+1}}. \]
Proposition~\ref{propstab1} can be applied with $r=r'$, provided
\begin{equation}\label{see}
\varepsilon \leq \frac{\beta'r'}{2^79K}=\frac{Mr'^2}{2^{10}}, \quad r'\leq r_0
\end{equation}
holds true. Choosing 
\[ K:=\left(\frac{\varepsilon_0'}{\varepsilon}\right)^{a'} \]
we have
\[ r'=R_*'\left(\frac{\varepsilon}{\varepsilon_0'}\right)^{\frac{1}{2}} \]
and the inequalities~\eqref{see} are satisfied if $\varepsilon \leq \min\{\varepsilon_0',\varepsilon_*'\}$, where $a'$, $\varepsilon_0'$, $\varepsilon_*'$ and $R_*'$ are the constants given in the statement. Then, exactly as before, Proposition~\ref{propstab1} yields
\[ ||(I(t),J(t))-(I_0,J_0)|| \leq R_*'\left(\frac{\varepsilon}{\varepsilon_0'}\right)^{\frac{1}{2}}, \quad |t|\leq \frac{s_0R_*'}{5\varepsilon_0'}\left(\frac{\varepsilon_0'}{\varepsilon}\right)^{\frac{1}{2}}\exp\left(\frac{s_0}{6}\left(\frac{\varepsilon_0'}{\varepsilon}\right)^{a'}\right), \]
which gives the first part of the statement.

Concerning the second part of the statement, as before it is enough to prove that the complement of $B_{\gamma',\tau'}$ has a relative measure of order $\gamma'$. The complement of $B_{\gamma',\tau'}$ is
\[ \left\{\omega \in \R^n\; | \; \exists (k,l)\in \Z^{n+m}, \quad |(k,l)\cdot (\omega,\alpha)| < \gamma' (|k|+|l|)^{-\tau'}\right\},  \]
but since $\alpha$ satisfies~\eqref{dio}, and since $\gamma'\leq\gamma$ and $\tau'\geq \tau$, this set is also equal to
\[ \left\{\omega \in \R^n\; | \; \exists (k,l)\in \Z^{n} \setminus \{0\} \times \Z^m, \quad |(k,l)\cdot (\omega,\alpha)| < \gamma' (|k|+|l|)^{-\tau'}\right\}.  \]
But now the above set is known to have a relative measure of order $\gamma'$: this is exactly the content of Lemma 2.12 in~\cite{Jor91}.
\end{proof}

\section{A more general stability result}\label{s7}

Let us finally give a more general result, where the Diophantine assumption~\eqref{dio} is removed. Assuming $\alpha \in \R^{m}$ to be simply non-resonant, we can define a function $\Psi=\Psi_{\alpha}$ by
\begin{equation}\label{funcpsi}
\Psi(K)=\max\left\{|k\cdot\alpha|^{-1} \; | \; k\in \Z^m, \; 0<|k|\leq K \right\}, \quad K \geq 1.
\end{equation} 
Then we define $\Delta=\Delta_{\alpha}$ by
\begin{equation}\label{funcdelta}
\Delta(x)=\sup\{K \geq 1\; | \; K\Psi(K)\leq x\}, \quad x \geq \Psi(1)=|\alpha|_\infty^{-1}, \quad |\alpha|_\infty:=\max_{1 \leq j \leq m}|\alpha_j|.
\end{equation}
If $\alpha$ satisfies~\eqref{dio}, then the functions $\Psi$ and $\Delta$ defined above satisfy
\begin{equation}\label{casdio}
\Psi(K) \leq \gamma^{-1}K^\tau, \quad \Delta(x)\geq (\gamma x)^{\frac{1}{\tau+1}}.
\end{equation}
The only place where~\eqref{dio} was used was in Lemma~\ref{nonresonant}. But using the function $\Psi$ instead, the exact same proof yields the following more general lemma.

\begin{lemma}\label{nonresonant2}
Let $E>0$ and $F \geq E+1$. Assume that $\alpha \in \R^m$ is non-resonant and 
\begin{equation*}
\begin{cases}
FK\lambda_d \leq \lambda_{d+1} \leq \Psi((d+1)K^{d+1})^{-1}, \quad 1 \leq d \leq n-1, \\ 
\lambda_n \leq F^{-1}\Psi((n+1)K^{n+1})^{-1}.
\end{cases}
\end{equation*}
Then for any $\Lambda \in M_{K}^a$, the block $B_\Lambda$ is $(\beta_\Lambda,K)$-non resonant modulo $\Lambda$ with
\begin{equation*}
\beta_\Lambda= EK\delta_{\Lambda}, \quad \Lambda\neq\{0\}, \quad \beta_{\{0\}}=\lambda_1.
\end{equation*}
\end{lemma}

Using this Lemma instead of Lemma~\ref{nonresonant}, we arrive at the following proposition which generalizes Proposition~\ref{covering}.

\begin{proposition}\label{covering2}
For $\Lambda \in M_K^a$ of rank $d$, with $0 \leq d \leq n$, let us define
\begin{equation*}
r_\Lambda:=\frac{1}{|\tilde{\Lambda}|F^{n-d+1}\Psi((n+1)K^{n+1})K^{n-d+1}},
\end{equation*}
where we set, by convention, $|\tilde{\Lambda}|=1$ if $\Lambda=\{0\}$. Then there exists a covering of $D$ by subsets $D_{\Lambda}$, where $\Lambda \in M_K^a$, such that each $D_\Lambda$ is $(\beta_\Lambda,K)$-non resonant modulo $\Lambda$ and, for $\Lambda \neq \{0\}$, $\delta_\Lambda$-close to $\Lambda$-resonances with
\begin{equation*}
\beta_\Lambda=\frac{9MKr_\Lambda}{8}, \quad \delta_\Lambda=\frac{mr_\Lambda}{8}.
\end{equation*}
\end{proposition}

Then, using Proposition~\ref{covering2} instead of Proposition~\ref{covering}, together with Proposition~\ref{propstab1} and Proposition~\ref{propstab2}, and proceeding exactly as in the Proof of Theorem~\ref{thm1}, we obtain the following statement.

\begin{theorem}\label{thm4}
Let $H$ be as in~\eqref{Ham}, with $h$ satisfying~\eqref{bound},~\eqref{bound2} and~\eqref{convex}, and $f$ satisfying~\eqref{pert}. Assume also that $\alpha$ is non-resonant, and that
\[ \varepsilon \leq \frac{(n+1)^2m|\alpha|_\infty
^2}{2^{10}F^{2(n+1)}}  \]
so that we can define
\[ \Delta_\varepsilon:=\Delta\left(\frac{n+1}{2^5F^{n+1}}\sqrt{\frac{m}{\varepsilon}}\right), \quad K_\varepsilon:=\left(\frac{\Delta_\varepsilon}{n+1}\right)^{\frac{1}{n+1}}, \quad R_\varepsilon:=\frac{1}{FK_\varepsilon\Psi((n+1)K_\varepsilon^{n+1})}. \] 
Then, if $K_\varepsilon \geq 1$ and $R_\varepsilon \leq r_0$, for any solution $(I(t),J(t),\theta(t),\varphi(t))$ of the system associated to $H$ with initial condition $(I_0,J_0,\theta_0,\varphi_0) \in \bar{D} \times \T^{n+m}$, we have
\[ ||I(t)-I_0||\leq R_\varepsilon, \quad |t|\leq T_*\exp\left(\frac{s_0 K_\varepsilon}{6}\right). \]
\end{theorem}

When $\alpha$ is Diophantine, then~\eqref{casdio} holds true and the above statement exactly reduces to Theorem~\ref{thm1}. There also analogues of Theorem~\ref{thm2} and Theorem~\ref{thm3} in the above context that can be stated and proved the same way; details are left to the reader. 

\section{Concluding remarks}\label{s8}

We conclude this paper with several remarks. We first discuss the Chirikov-Lochak conjecture, and then the possibility of extending our results to less regular Hamiltonians, more general classes of integrable Hamiltonians and finally more general time-dependence.

First let us recall that the conjecture of Chirikov-Lochak predicts that
\begin{equation}\label{expchi2}
a=b=\frac{1}{2(n+1+\tau)} 
\end{equation}
while we proved 
\[ a=\frac{1}{2(n+1)(1+\tau)}, \quad b=\frac{(n+1)\tau+1}{2(n+1)(1+\tau)}. \]
For $\tau=0$, the exponents are the same but not for $\tau\geq 1$. However there is certainly no contradiction here; our result yields a better confinement but on a worse time-scale. Chirikov initial conjecture is ultimately based on the \textit{ansatz}
\begin{equation}\label{ansatz}
|(k,l)\cdot(\omega,\alpha)|\sim \frac{1}{K^{n+m-1}}, \quad |k|+|l|\leq K,
\end{equation}
for an arbitrary fixed vector $\alpha\in\R^m$, leading to the exponents
\begin{equation}\label{expchi}
a=b=\frac{1}{2(n+m)}.
\end{equation} 
We refer to~\cite{CV96}, Equation $(1.8)$ for instance.  When $m=0$ or $m=1$, that is when the perturbation is autonomous or time-periodic, vectors $\omega \in \R^n$ satisfying the above condition~\eqref{ansatz} do exist: they are called badly approximable, they are the ``best" non-resonant vectors and they form a dense set but of zero Lebesgue measure. In those two cases $m=0$ and $m=1$, the exponents~\eqref{expchi} turned out to be correct, but proofs do not use existence (nor density) of those badly approximable vectors: on the contrary, in the work of Lochak it is the ``worst" resonant vectors, namely periodic vectors, that play the major role. When $m \geq 2$, as observed by Lochak in \cite{LMS}, the ansatz~\eqref{ansatz} and hence the exponents~\eqref{expchi} are clearly too optimistic: one should require $\alpha$ to be Diophantine with some exponent $\tau \geq m-1$, and replace~\eqref{ansatz} by
\begin{equation}\label{ansatz2}
|(k,l)\cdot(\omega,\alpha)|\sim \frac{1}{K^{n+\tau}}, \quad |k|+|l|\leq K,
\end{equation}
so that the exponents of~\eqref{expchi} become those of~\eqref{expchi2}. Vectors $\omega \in \R^n$ satisfying~\eqref{ansatz2} do exist, they form a dense set which is furthermore of full Lebesgue measure when $\tau>m-1$ (when $\tau=m-1$ they have zero Lebesgue measure). In our opinion, it is a very interesting open problem to derive first these exponents when $m=0$ or $m=1$ using badly approximable vectors, to see if and how one can try to use vectors satisfying~\eqref{ansatz2} to possibly reach the values~\eqref{expchi2}. 

Then, let us comment on the regularity assumption in the results we proved. The assumptions that $h$ and $f$ are real-analytic were only used in the normal form Lemma~\ref{normal}, which is taken from~\cite{Pos93}. So in order to have results for Hamiltonians which are not real-analytic, for instance Hamiltonians which are only Gevrey regular or finitely differentiable, one just needs a version of Lemma~\ref{normal} in those settings, and this appears to be only a problem of technical nature. One can find in \cite{Bou13a} and \cite{Bou13b} normal form results for non-analytic Hamiltonians in the spirit of Lemma~\ref{normal}, even though those statements do not recover Lemma~\ref{normal}. Let us also recall that in the autonomous case (or time-periodic case when $h$ is convex), Nekhoroshev type estimates are known for Gevrey or finitely differentiable Hamiltonians (\cite{MS02}, \cite{Bou10}, \cite{Bou11}) but all those proofs are based on the Lochak method. 

Next let us discuss the more interesting question of whether our results extend to steep integrable Hamiltonians, the original class of integrable Hamiltonians considered by Nekhoroshev. One should first recall that in the general steep case, unlike what happens in the convex case, motions near resonances are not necessarily confined: if they are not, the steepness property ensures that they evolve towards a less resonant domain and, eventually, end up in a non-resonant domain on which an integrable normal form (up to a small remainder) can be constructed, leading to the stability of the action variables. The issue is that for a quasi-periodic perturbation, the non-resonant domain might be empty so the extension to steep integrable Hamiltonians does not seem straightforward. Indeed, consider in fact a time-periodic perturbation of the simplest class of steep Hamiltonians, namely quasi-convex Hamiltonians. Furthermore let's look at the simplest quasi-convex Hamiltonian, which is certainly
\[ h(I_1,\dots,I_n)=\frac{1}{2}(I_1^2+\cdots+I_{n-1}^2)+I_n. \] 
After a periodic perturbation the extended integrable Hamiltonian then reads
\[ \bar{h}(I_1,\dots,I_n,J)=\frac{1}{2}(I_1^2+\cdots+I_{n-1}^2)+I_n+J \] 
and frequencies associated to $\bar{h}$ are all resonant (as the last two components of any frequency vectors are both equal to $1$). Thus for such a Hamiltonian, integrable normal forms cannot be obtained (and even ``almost" integrable normal forms, in which the dependence on the angle associated to $J$ is allowed, as one can easily check) and to prove stability, if any, one should follow a different path than the one of Nekhoroshev.

To conclude, let us mention that it seems quite unlikely to have a non-trivial stability result for an arbitrary time-dependent perturbation, unless the time depends on the small parameter, in which case the conjugated action variable can be considered as degenerate (see \cite{GZ92} or \cite{Bou13CRAS}). But it may be possible to extend our results for a class of perturbation whose Fourier transform (with respect to time) has suitable localization properties.

\addcontentsline{toc}{section}{References}
\bibliographystyle{amsalpha}
\bibliography{NekhoQP2}

\end{document}